\documentclass[reqno,10pt, centertags]{amsart}
\usepackage{amssymb,upref,esint,float,color}
\usepackage{centernot}
\usepackage{hyperref}
\newcommand*{\mailto}[1]{\href{mailto:#1}{\nolinkurl{#1}}}

\makeatletter
\def\theequation{\@arabic\c@equation}

\newcommand{\bbN}{{\mathbb{N}}}
\newcommand{\bbR}{{\mathbb{R}}}

\newcommand{\bbZ}{{\mathbb{Z}}}
\newcommand{\bbC}{{\mathbb{C}}}

\newcommand{\cB}{{\mathcal B}}
\newcommand{\cC}{{\mathcal C}}

\newcommand{\cF}{{\mathcal F}}
\newcommand{\cH}{{\mathcal H}}

\newcommand{\cK}{{\mathcal K}}

\newcommand{\no}{\nonumber}
\newcommand{\lb}{\label}
\newcommand{\f}{\frac}

\newcommand{\ol}{\overline}

\newcommand{\Oh}{O}

\newcommand{\tr}{\text{\rm{tr}}}

\newcommand{\rank}{\text{\rm{rank}}}
\newcommand{\ran}{\text{\rm{ran}}}

\newcommand{\dom}{\text{\rm{dom}}}

\newcommand{\bi}{\bibitem}
\newcommand{\sgn}{\text{\rm{sign}}}

\renewcommand{\Re}{\text{\rm Re}}
\renewcommand{\Im}{\text{\rm Im}}
\renewcommand{\ln}{\text{\rm ln}}

\DeclareMathOperator{\SL}{SL}

\renewcommand{\dot}{\overset{\textbf{\Large.}}}

\numberwithin{equation}{section}

\newtheorem{theorem}{Theorem}[section]
\newtheorem{lemma}[theorem]{Lemma}

\newtheorem{hypothesis}[theorem]{Hypothesis}
\newtheorem{example}[theorem]{Example}
\theoremstyle{remark}
\newtheorem{remark}[theorem]{Remark}

\begin{document}

\title[Computation of Traces, Determinants, and $\zeta$-Functions]{Effective Computation
of Traces, Determinants, and $\zeta$-Functions for Sturm--Liouville Operators}

\author[F.\ Gesztesy]{Fritz Gesztesy}
\address{Department of Mathematics,
Baylor University, One Bear Place \#97328,
Waco, TX 76798-7328, USA}
\email{\mailto{Fritz\_Gesztesy@baylor.edu}}
%\email{Fritz$\_$Gesztesy@baylor.edu}
\urladdr{\url{http://www.baylor.edu/math/index.php?id=935340}}
%\urladdr{http://www.baylor.edu/math/index.php?id=935340}

\author[K.\ Kirsten]{Klaus Kirsten}
\address{GCAP-CASPER, Department of Mathematics,
Baylor University, One Bear Place \#97328,
Waco, TX 76798-7328, USA}
\email{\mailto{Klaus\_Kirsten@baylor.edu}}
%\email{Klaus$\_$Kirsten@baylor.edu}
\urladdr{\url{http://www.baylor.edu/math/index.php?id=54012}}
%\urladdr{http://www.baylor.edu/math/index.php?id=54012}

%%%%%%%%%%%%%%%%%%%%%%%%%%%%%%%%%%%%%%%%%
\dedicatory{Dedicated with great pleasure to Konstantin Makarov on the occasion of his
60th birthday.}
\date{\today}
%\date{, 2003.}
\thanks{K.K.\ was supported by the Baylor University Summer Sabbatical Program.}
\thanks{Appeared in {\it J. Funct. Anal.} {\bf 276}, 520--562 (2019).}

%%%%%%%%%%%%%%%%%%%%%%%%%%%%%%%%%%%%%%%%%
\begin{abstract}
The principal aim in this paper is to develop an effective and unified approach to the computation of
traces of resolvents (and resolvent differences), Fredholm determinants, $\zeta$-functions, and
$\zeta$-function regularized determinants associated with linear operators in a Hilbert space. In
particular, we detail the connection between Fredholm and $\zeta$-function regularized determinants.

Concrete applications of our formalism to general (i.e., three-coefficient) regular
Sturm--Liouville operators on bounded intervals with various (separated and coupled) boundary
conditions, and Schr\"odinger operators on a half-line, are provided and further illustrated with
an array of examples.
\end{abstract}
%%%%%%%%%%%%%%%%%%%%%%%%%%%%%%%%%%%%%%%%%

\maketitle

%\newpage

%{\scriptsize{\tableofcontents}}
%\normalsize

%%%%%%%%%%%%%%%%%%%%%%%%%%%%%%%%%%%%%%%%%
%%%%%%%%%%%%%%%%%%%%%%%%%%%%%%%%%%%%%%%%%
\section{Introduction} \lb{14.s1}
%%%%%%%%%%%%%%%%%%%%%%%%%%%%%%%%%%%%%%%%%
%%%%%%%%%%%%%%%%%%%%%%%%%%%%%%%%%%%%%%%%%

{\it We dedicate this paper with great pleasure to Konstantin A.\ Makarov on the occasion of his 60th birthday. Happy Birthday, Konstantin, we hope our modest contribution to spectral theory
will cause some joy.}

\smallskip

The prime motivation for writing this paper originally was to detail the relationship between Fredholm
and $\zeta$-function regularized determinants of Hilbert space operators and providing a unified approach to these determinants as well as traces of resolvents.

The case of traces of resolvents (and resolvent differences), Fredholm determinants, $\zeta$-functions, and $\zeta$-function regularized determinants associated with linear operators in a Hilbert space is described in detail in Section \ref{14.s2}. In particular, under an appropriate trace class hypothesis on
the resolvent of a self-adjoint operator $S$ in $\cH$ we define its $\zeta$-function and derive a formula
for it in terms of the self-adjoint functional calculus and the resolvent of $S$. In the case of trace class resolvent differences for a pair of self-adjoint operators $(S_1,S_2)$ in $\cH$ we also describe the underlying relative spectral $\zeta$-function for such a pair and relate the corresponding
$\zeta$-function regularized relative determinant and a symmetrized Fredholm perturbation determinant.

In Section \ref{14.s3} we provide an exhaustive discussion of regular (three-coefficient) self-adjoint
Sturm--Liouville operators, that is, operators in $ L^2((a,b); rdx)$ associated with self-adjoint realizations of differential expressions of the type $\tau = r^{-1} [-(d/dx) p (d/dx) + q]$, with arbitrary (separated and coupled) self-adjoint boundary conditions on compact intervals $[a,b]$. Their traces of resolvents and associated perturbation determinants
are calculated for all self-adjoint boundary conditions in terms of concrete expressions involving a
canonical system of fundamental solutions $\phi(z,\, \cdot\,,a)$ and $\theta(z,\, \cdot \,,a)$ of
$\tau \psi = z \psi$, and special examples such as Floquet boundary conditions and the Krein--von Neumann extension are highlighted. The $\zeta$-function regularized determinants are determined for all self-adjoint
boundary conditions and a variety of concrete examples complete this section.

Our final Section \ref{14.s4} then illustrates some of the abstract notions in
Section \ref{14.s2} with the help of self-adjoint Schr\"odinger operators associated with
differential expressions of the type $-(d^2/dx^2) + q$ on the half-line $\bbR_+ =(0,\infty)$ with
short-range potentials $q$ (i.e., we treat the scattering theory situation). Again we study all self-adjoint boundary conditions at $x=0$. The assumed short-range nature of $q$ then necessitates a comparison with the case $q = 0$ and thus illustrates the case of relative perturbation determinants, relative
$\zeta$-functions, and relative $\zeta$-function regularized determinants abstractly discussed in
Section \ref{14.s2}.

Finally, we summarize some of the basic notation used in this paper (especially, in Section \ref{14.s2}): Let $\cH$ and $\cK$
be separable complex Hilbert spaces, $(\,\cdot\,,\,\cdot\,)_{\cH}$ and
$(\,\cdot\,,\,\cdot\,)_{\cK}$ the scalar products in $\cH$ and $\cK$ (linear in
the second factor), and
$I_{\cH}$ and $I_{\cK}$ the identity operators in $\cH$ and $\cK$,
respectively.

Next, let $T$ be a closed linear operator from
$\dom(T)\subseteq\cH$ to $\ran(T)\subseteq\cK$, with $\dom(T)$
and $\ran(T)$ denoting the domain and range of $T$. The closure of a
closable operator $S$ is denoted by $\ol S$. The kernel (null space) of $T$
is denoted by $\ker(T)$.

The spectrum, point spectrum, and resolvent set of a closed linear
operator in $\cH$ will be denoted by $\sigma(\cdot)$, $\sigma_p(\cdot)$, and $\rho(\cdot)$; the
discrete spectrum of $T$ (i.e., points in $\sigma_p(T)$ which are isolated from the rest of
$\sigma(T)$, and which are eigenvalues of $T$ of finite algebraic multiplicity) is
abbreviated by $\sigma_d(T)$. The {\it algebraic multiplicity} $m_a(z_0; T)$ of an eigenvalue
$z_0\in\sigma_d(T)$ is the dimension of the range of the corresponding {\it Riesz projection}
$P(z_0;T)$,
\begin{equation}
m_a(z_0; T) = \dim(\ran(P(z_0;T))) = \tr_{\cH}(P(z_0;T)),
\end{equation}
where (with the symbol $ \ointctrclockwise$ denoting counterclockwise oriented
contour integrals)
\begin{equation}
 P(z_0;T)=\f{-1}{2\pi i} \ointctrclockwise_{C(z_0;\varepsilon)} d\zeta \,
(T - \zeta I_{\cH})^{-1},
\end{equation}
for $0 < \varepsilon<\varepsilon_0$ and  $D(z_0;\varepsilon_0) \backslash \{z_0\}\subset \rho(T)$; here
$D(z_0; r_0) \subset \bbC$ is the open disk with center $z_0$ and radius
$r_0 > 0$, and $C(z_0; r_0) = \partial D(z_0; r_0)$ the corresponding circle.
The  {\it geometric multiplicity} $m_g(z_0; T)$ of an eigenvalue
$z_0 \in \sigma_p(T)$ is defined by
\begin{equation}
m_g(z_0; T) = \dim(\ker((T - z_0 I_{\cH}))).
\end{equation}
The essential spectrum of $T$ is defined by $\sigma_{ess}(T) = \sigma(T)\backslash \sigma_d(T)$.

The Banach spaces of bounded
and compact linear operators in $\cH$ are denoted by $\cB(\cH)$ and
$\cB_\infty(\cH)$, respectively. Similarly, the Schatten--von Neumann
(trace) ideals will subsequently be denoted by $\cB_p(\cH)$,
$p \in [1,\infty)$, and the subspace of all finite rank operators in $\cB_1(\cH)$ will be
abbreviated by $\cF(\cH)$. Analogous notation $\cB(\cH_1,\cH_2)$,
$\cB_\infty (\cH_1,\cH_2)$, etc., will be used for bounded, compact, etc.,
operators between two Hilbert spaces $\cH_1$ and $\cH_2$. In addition,
$\tr_{\cH}(T)$ denotes the trace of a trace class operator $T\in\cB_1(\cH)$,
${\det}_{\cH} (I_{\cH} - T)$ the Fredholm determinant of $I_{\cH} - T$.
%and for $p \in \bbN$, $p \geq 2$, ${\det}_{\cH,p} (I_{\cH} - T)$ abbreviates the
%$p$th modified Fredholm determinant of $I_{\cH} - T$.

%The symbol $\dotplus$ denotes a direct (but not necessary orthogonal direct)
% decomposition in connection with subspaces of Banach spaces.
Finally, we find it convenient to abbreviate $\bbN_0 = \bbN \cup \{0\}$.

%%%%%%%%%%%%%%%%%%%%%%%%%%%%%%%%%%%%%%%%%%
%%%%%%%%%%%%%%%%%%%%%%%%%%%%%%%%%%%%%%%%%%
\section{Traces, Fredholm Determinants, and Zeta Functions of Operators} \lb{14.s2}
%%%%%%%%%%%%%%%%%%%%%%%%%%%%%%%%%%%%%%%%%%
%%%%%%%%%%%%%%%%%%%%%%%%%%%%%%%%%%%%%%%%%%

In this section we recall some well-known formulas relating traces and Fredholm
determinants and also discuss the notion of $\zeta$-functions of self-adjoint operators. For background
on the material used in this section see, for instance, \cite{GGK96}, \cite{GGK97}, \cite[Ch.\ XIII]{GGK00},
\cite[Ch.~IV]{GK69}, \cite{Ku61}, \cite[Ch.\ 17]{RS78}, \cite{Si77}, \cite[Ch.\ 3]{Si05}.

To set the stage we start with densely defined, closed, linear operators $A$ in $\cH$ having a trace class resolvent, and hence introduce the following assumption:

%%%%%%
\begin{hypothesis} \lb{14.h2.1}
Suppose that $A$ is densely defined and closed in $\cH$ with $\rho(A) \neq \emptyset$, and
$(A - z I_{\cH})^{-1} \in \cB_1(\cH)$ for some $($and
hence for all\,$)$ $z \in \rho(A)$.
\end{hypothesis}
%%%%%%

Given Hypothesis \ref{14.h2.1} and $z_0 \in \rho(A)$, consider the bounded, entire family $A(\, \cdot \,)$ defined by
\begin{equation}
A(z) := I_{\cH} - (A - z I_{\cH})(A - z_0 I_{\cH})^{-1} = (z - z_0) (A - z_0 I_{\cH})^{-1}, \quad
z \in \bbC.     \lb{14.2.1}
\end{equation}
Employing the formula (cf.\ \cite[Sect.~IV.1]{GK69}, see also \cite{Ku61}, \cite[Sect.~I.7]{Ya92}),
\begin{equation}
\tr_{\cH}\big((I_{\cH} - T(z))^{-1} T'(z)\big) = - (d/dz) \ln({\det}_{\cH}(I_{\cH} - T(z))),   \lb{14.2.2}
\end{equation}
valid for a trace class-valued analytic family $T(\, \cdot \,)$ on an open set $\Omega \subset \bbC$
such that $(I_{\cH} - T(\, \cdot \,))^{-1} \in \cB(\cH)$, and applying it to the entire family $A(\, \cdot \,)$ then results in
\begin{align}
\tr_{\cH}\big((A - z I_{\cH})^{-1}\big) &= - (d/dz) \ln\big({\det}_{\cH}\big(I_{\cH} - (z - z_0) (A - z_0 I_{\cH})^{-1}\big)\big)   \no \\
&= - (d/dz) \ln\big({\det}_{\cH}\big((A - z I_{\cH}) (A - z_0 I_{\cH})^{-1}\big)\big),   \lb{14.2.3} \\
& \hspace*{5.75cm} z \in \rho(A).    \no
\end{align}

One notes that the left- and hence the right-hand side of \eqref{14.2.3} is independent of the choice
of $z_0 \in \rho(A)$.

Next, following the proof of \cite[Theorem~3.5\,(c)]{Si05} step by step, and employing a
Weierstrass-type product formula (see, e.g., \cite[Theorem~3.7]{Si05}), yields the subsequent
result (see also \cite{GW95}).

%%%%%%
\begin{lemma} \lb{14.l2.2}
Assume Hypothesis \ref{14.h2.1} and let $\lambda_k \in \sigma(A)$ then
\begin{equation} \lb{14.2.4}
{\det}_{\cH} \big(I_{\cH} - (z-z_0)(A - z_0 I_{\cH})^{-1}\big)
\underset{z \to \lambda_k}{=} (\lambda_k-z)^{m_a(\lambda_k)}
[C_k +O(\lambda_k - z)], \quad C_k\neq 0,
\end{equation}
that is, the multiplicity of the zero
 of the Fredholm determinant ${\det}_{\cH}\big(I_{\cH} - (z-z_0)(A - z_0 I_{\cH})^{-1}\big)$
at $z=\lambda_k$ equals the algebraic multiplicity of the eigenvalue $\lambda_k$ of $A$.

In addition, denote the spectrum of $A$
by $\sigma(A)=\{\lambda_k\}_{k\in\bbN}$, $\lambda_k \neq \lambda_{k'}$
for $k \neq k'$. Then
\begin{align}
\begin{split}
{\det}_{\cH}(I_{\cH} - (z-z_0)(A - z_0 I_{\cH})^{-1}) &= \prod_{k \in \bbN}
\big[1-(z-z_0)(\lambda_k - z_0)^{-1}\big]^{m_a(\lambda_k)} \\
&= \prod_{k \in \bbN} \bigg(\f{\lambda_k - z}{\lambda_k-z_0}
\bigg)^{m_a(\lambda_k)},    \lb{14.2.5}
\end{split}
\end{align}
with absolutely convergent products in \eqref{14.2.5}.
\end{lemma}
%%%%%%%

In many cases of interest not a single resolvent, but a difference of two
resolvents is trace class and hence one is naturally led to the following generalization discussed in detail in \cite{GLMZ05}, \cite{GN12a},
\cite{GZ12} (see also \cite{GN12}). To avoid technicalities, we will now consider the case of self-adjoint operators $A, B$ below, but note that \cite{GLMZ05} considers the general case of densely defined, closed linear operators with nonempty resolvent sets:

%%%%%%%%%%%
\begin{hypothesis} \lb{14.h2.4}
Suppose $A$ and $B$ are self-adjoint operators in $\cH$ with $A$ bounded from below. \\
$(i)$ Assume that $B$ can be written as the form sum $($denoted by $+_q$$)$
of $A$ and a self-adjoint operator $W$ in $\cH$
\begin{equation}
B = A +_q W,     \lb{14.2.1A}
\end{equation}
where $W$ can be factorized into
\begin{equation}
W = W_1^* W_2,    \lb{14.2.2A}
\end{equation}
such that
\begin{equation}
\dom(W_j) \supseteq \dom\big(|A|^{1/2}\big), \quad j =1,2.   \lb{14.2.2aA}
\end{equation}
$(ii)$ Suppose that for some $($and hence for all\,$)$ $z \in \rho(A)$,
\begin{equation}
W_2 (A - z I_{\cH})^{-1/2}, \, \ol{(A - z I_{\cH})^{-1/2} W_1^*} \in \cB_2(\cH) .    \lb{14.2.3A}
\end{equation}
\end{hypothesis}
%%%%%%%%%%%

\medskip

Given Hypothesis \ref{14.h2.4}, one observes that
\begin{equation}
\dom\big(|B|^{1/2}\big) = \dom\big(|A|^{1/2}\big),    \lb{14.2.5A}
\end{equation}
and that the resolvent of $B$ can be written as (cf., e.g., the detailed discussion in
\cite{GLMZ05} and the references therein)
\begin{align}
(B - z I_{\cH})^{-1} &= (A - z I_{\cH})^{-1}    \no \\
& \quad - \ol{(A - z I_{\cH})^{-1} W_1^*}
\big[I_{\cH} + \ol{W_2 (A - z I_{\cH})^{-1} W_1^*}\big]^{-1} W_2 (A - z I_{\cH})^{-1},
\no\\
& \hspace*{6.7cm} z \in \rho(B) \cap \rho(A).    \lb{14.2.6A}
\end{align}
We also note the analog of Tiktopoulos' formula (cf.\ \cite[p.~45]{Si71}),
\begin{align}
\begin{split}
(B - z I_{\cH})^{-1} &= (A - z I_{\cH})^{-1/2}
\big[I_{\cH} + \ol{(A - z I_{\cH})^{-1/2} W (A - z I_{\cH})^{-1/2}}\big]^{-1}   \\
& \quad \times (A - z I_{\cH})^{-1/2}, \quad z \in \rho(B) \cap \rho(A).
\lb{14.2.7A}
\end{split}
\end{align}
Here the closures in \eqref{14.2.6A} and \eqref{14.2.7A} are well-defined
employing \eqref{14.2.3A}. In addition, one observes that the resolvent formulas \eqref{14.2.6A} and \eqref{14.2.7A} are symmetric with respect
to $A$ and $B$ employing $A = B -_q W$.

As a consequence, $B$ is bounded from below in $\cH$ and one concludes that
for some $($and hence for all\,$)$ $z \in \rho(B) \cap \rho(A)$,
\begin{equation}
\big[(B - z I_{\cH})^{-1} - (A - z I_{\cH})^{-1}\big] \in \cB_1(\cH).  \lb{14.2.4A}
\end{equation}

Moreover, one infers that (cf.\ \cite{GZ12})
\begin{align}
& {\tr}_{\cH}\big((B - z I_{\cH})^{-1} - (A - z I_{\cH})^{-1}\big)   \no \\
& \quad
= - \f{d}{dz} \ln\Big({\det}_{\cH}\Big(\ol{(B - z I_{\cH})^{1/2}(A - z I_{\cH})^{-1}
	(B - z I_{\cH})^{1/2}}\Big)\Big)    \lb{14.2.8B} \\
&\quad =  - \f{d}{dz} \ln\Big({\det}_{\cH}\Big(I_{\cH} +\ol{W_2 (A - z I_{\cH})^{-1}
W_1^*}\Big)\Big), \quad z \in \rho(B) \cap \rho(A).     \lb{14.2.8A}
\end{align}
Here any choice of branch cut of the normal operator
$(B - z I_{\cH})^{1/2}$ (employing the spectral theorem) is permissible.
The first equality in \eqref{14.2.8A} follows as in the proof of
\cite[Theorem~2.8]{GZ12}. (The details are actually a bit simpler now
since $A, B$ are self-adjoint and bounded from below, and hence of positive type after some
translation, which is the case considered in \cite{GZ12}). For completeness, we mention that the second
equality in \eqref{14.2.8A} can be arrived at as follows: Employing the commutation formula (cf., \cite{De78}),
\begin{equation}
C[I_{\cH} - DC]^{-1} D = - I_{\cH} + [I_{\cH} - CD]^{-1}
\end{equation}
for $C, D \in \cB(\cH)$ with $1 \in \rho(DC)$ (and hence $1 \in \rho(CD)$
since $\sigma(CD)\backslash \{0\} = \sigma(DC)\backslash \{0\}$),
the resolvent identity \eqref{14.2.7A} with $A$ and $B$ interchanged yields
\begin{align}
& \ol{(B - z I_{\cH})^{1/2}(A - z I_{\cH})^{-1} (B - z I_{\cH})^{1/2}}   \no \\
& \quad = \big[I_{\cH} - \ol{(B - z I_{\cH})^{-1/2}
W (B - z I_{\cH})^{-1/2}}\big]^{-1}  \no \\
& \quad = \Big[I_{\cH} - \ol{\big[(B - z I_{\cH})^{-1/2} W_1^*}\big]
\big[W_2 (B - z I_{\cH})^{-1/2}\big]\Big]^{-1}.
\end{align}
On the other hand, it is well-known (cf.\ \cite{GLMZ05}, \cite{Ka66}) that
\begin{align}
& I_{\cH} + \ol{W_2 (A - z I_{\cH})^{-1} W_1^*}   \no \\
& \quad =\big[I_{\cH} - \ol{W_2 (B - z I_{\cH})^{-1} W_1^*}\big]^{-1}   \no \\
& \quad = \Big[I_{\cH} - \big[W_2 (B - z I_{\cH})^{-1/2}\big]
\ol{\big[(B - z I_{\cH})^{-1/2} W_1^*\big]}\Big]^{-1},
\end{align}
and hence using the fact
\begin{equation}
{\det}_{\cH} (I_{\cH} - ST) = {\det}_{\cH} (I_{\cH} - TS)  \lb{14.2.8AA}
\end{equation}
for $S, T \in \cB(\cH)$ with $ST, TS \in \cB_1(\cH)$ (again, since
$\sigma(ST)\backslash \{0\} = \sigma(TS)\backslash \{0\}$) then proves the
second equality in \eqref{14.2.8A}.

While this approach based on Hypothesis \ref{14.h2.4} is tailor-made for a discussion of perturbations
of the potential coefficient in the context of Schr\"odinger and, more generally, Sturm--Liouville operators, we also mention the following variant that is best suited for changes in the boundary conditions:

%%%%%%%%%%%%
\begin{hypothesis} \lb{14.h2.4a}
Suppose that $A$ and $B$ are self-adjoint operators in $\cH$
bounded from below. In addition, assume that
\begin{align}
& \dom\big(|A|^{1/2}\big) \subseteq
\dom\big(|B|^{1/2}\big),     \lb{14.2.21a} \\
\begin{split}
& \ol{(B - t I_{\cH})^{1/2} \big[(B - t I_{\cH})^{-1}
- (A - t I_{\cH})^{-1}\big]
	(B - t I_{\cH})^{1/2}} \in \cB_1(\cH)    \lb{14.2.22a} \\
& \quad \text{for some $t < \inf(\sigma(A) \cup \sigma(B))$.}
\end{split}
\end{align}
\end{hypothesis}
%%%%%%%%%%%%

Given Hypothesis \ref{14.h2.4a} it has been proven in \cite{GZ12}
(actually, in a more general context) that for
$z \in \bbC \backslash [\inf(\sigma(A) \cup \sigma(B)),\infty)$,
\begin{align}
& \ol{(B - z I_{\cH})^{1/2} \big[(B - z I_{\cH})^{-1} - (A - z I_{\cH})^{-1}\big]
	(B - z I_{\cH})^{1/2}}     \no \\
	& \quad = I_{\cH} - \ol{(B - z I_{\cH})^{1/2} (A - z I_{\cH})^{-1}	
(B - z I_{\cH})^{1/2}}       \lb{14.2.23a} \\
& \quad = I_{\cH} - \big[(B - z I_{\cH})^{1/2} (A - z I_{\cH})^{-1/2}\big]	
\big[(B - {\ol z} I_{\cH})^{1/2} (A - {\ol z} I_{\cH})^{-1/2}\big]^*   \no
\end{align}
implying that ${\det}_{\cH}\big(\ol{(B - z I_{\cH})^{1/2}(A - z I_{\cH})^{-1}
(B - z I_{\cH})^{1/2}}\big)$ is well-defined. Moreover,
\begin{align}
& {\tr}_{\cH}\big((B - z I_{\cH})^{-1} - (A - z I_{\cH})^{-1}\big)   \no \\
& \quad = - \f{d}{dz} \ln\Big({\det}_{\cH}\Big(\ol{(B - z I_{\cH})^{1/2}(A - z I_{\cH})^{-1}
	(B - z I_{\cH})^{1/2}}\Big)\Big),   \lb{14.2.24a} \\
& \hspace*{5cm }
 z \in \bbC \backslash [\inf(\sigma(A) \cup \sigma(B)),\infty).    \no
\end{align}

Next we briefly turn to spectral $\zeta$-functions of self-adjoint operators $S$ with a trace class resolvent (and hence purely discrete spectrum).

%%%%%%%
\begin{hypothesis} \lb{14.h2.5}
Suppose $S$ is a self-adjoint operator in $\cH$, bounded from below, satisfying
\begin{equation}
(S- z I_{\cH})^{-1} \in \cB_1(\cH)
\end{equation}
for some $($and hence for all\,$)$ $z \in \rho(S)$.
We denote the spectrum of $S$ by $\sigma(S) = \{\lambda_j\}_{j \in J}$ $($with $J \subset \bbZ$ an appropriate index set\,$)$, with every eigenvalue repeated according to its multiplicity. 
\end{hypothesis}
%%%%%%%

Given Hypothesis \ref{14.h2.5}, the spectral
zeta function of $S$ is then defined by
\begin{equation}
\zeta (s; S) := \sum_{\substack{j \in J \\ \lambda_j \neq 0}} \lambda_j^{-s}    \lb{14.2.8}
\end{equation}
for $\Re(s) > 0$ sufficiently large such that \eqref{14.2.8} converges absolutely.

Next, let $P(0; S)$ be the spectral projection of $S$ corresponding to
the eigenvalue $0$ and denote by $m(\lambda_0; S)$ the multiplicity of the eigenvalue $\lambda_0$ of $S$, in particular,
\begin{equation}
m(0; S) = \dim(\ker(S)).
\end{equation}
(One recalls that since $S$ is self-adjoint, the geometric and algebraic multiplicity of each eigenvalue of $S$ coincide and hence the subscript ``$g$'' or ``$a$'' is simply omitted from $m(\, \cdot \,; S)$.) In addition, we introduce the simple contour $\gamma$ encircling $\sigma(S) \backslash \{0\}$  in a 
counterclockwise manner so as to dip under (and hence avoid) the point $0$ (cf.\ Figure \ref{14.2.fig1}). In fact, following \cite{KM04} (see also \cite{KM03}), we will henceforth choose as the branch cut of $z^{-s}$ the ray
\begin{equation}
R_{\theta} = \big\{z = t e^{i \theta} \big| t \in [0,\infty)\big\}, \quad
\theta \in (\pi/2, \pi),  \lb{14.2.10}
\end{equation}
and note that the contour $\gamma$ avoids any contact with $R_{\theta}$
(cf.\ Figure \ref{14.2.fig1}).
\begin{figure}[H]%[ht]
\setlength{\unitlength}{1cm}

\begin{center}

\begin{picture}(20,10)(0,0)

\put(0,0){\setlength{\unitlength}{1.0cm}
\begin{picture}(10,7.5)
\thicklines

\put(0,4){\vector(1,0){10}} \put(5.0,0){\vector(0,1){8}}
\put(10.0,4){\oval(9.4,1)[tl]} \put(5,4){\oval(0.6,0.5)[b]}
\put(4.4,4){\oval(0.6,1)[tr]} \put(4.4,4){\oval(3.0,1)[l]}
\put(7.5,4.5){\vector(-1,0){.4}}\put(4.4,3.5){\line(1,0){5.6}}
\put(5.0,4.0){\line(-1,2){2.0}} \put(.7,8.2){{\bf The cut $R_{\theta}$ for
$z^{-s}$}} \multiput(5.6,4)(.4,0){10}{\circle*{.15}}
\multiput(3.5,4)(.4,0){3}{\circle*{.15}}
%\multiput(5.0,4.5)(0,.3){5}{\circle*{.15}}
%\put(4.9,6){\line(1,0){.2}} \put(5.2,5.9){m}
\put(8.0,7){{\bf $z$-plane}}
\put(7.15,4.8){{\bf $\gamma$}}

\end{picture}}

\end{picture}

\caption{Contour $\gamma$ in the complex $z$-plane.} \lb{14.2.fig1}

\end{center}

\end{figure}
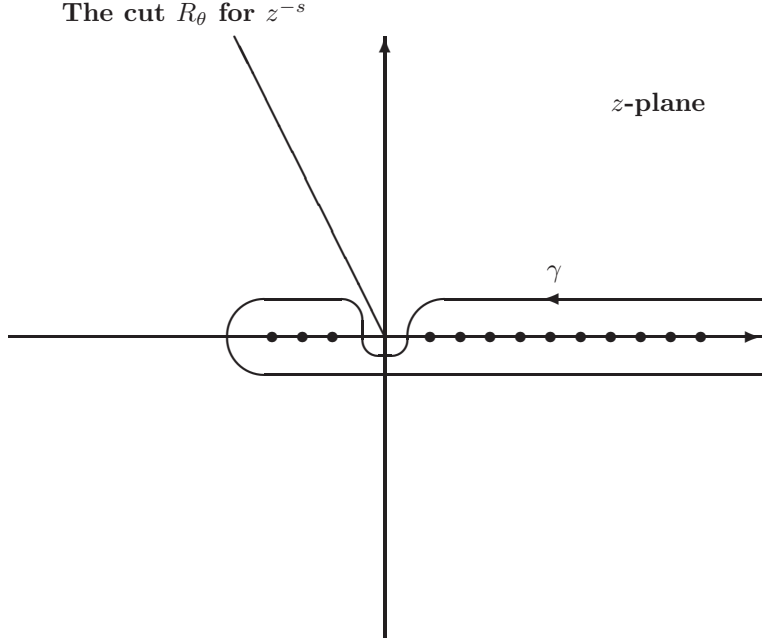

We note in passing that one could also use a semigroup approach via
\begin{align}
\begin{split}
\zeta(s; S) &= \Gamma(s)^{-1} \int_0^{\infty} dt \, t^{s-1} \tr_{\cH}\big(e^{- t S}
[I_{\cH} - P(0; S)]\big)   \lb{14.2.12} \\
&= \Gamma(s)^{-1} \int_0^{\infty} dt \, t^{s-1} \big[\tr_{\cH}\big(e^{- t S}\big) - m(0; S)\big],
\end{split}
\end{align}
for $\Re(s) > 0$ sufficiently large, but we prefer to work with resolvents in this paper.

%%%%%%%
\begin{lemma} \lb{14.l2.6}
In addition to Hypothesis \ref{14.h2.5} and the counterclockwise oriented contour $\gamma$ just described $($cf.\ Figure \ref{14.2.fig1}$)$, suppose that
$\big|{\tr}_{\cH}\big((S- z I_{\cH})^{-1} [I_{\cH} - P(0; S)]\big)\big|$
is polynomially bounded on $\gamma$. Then
\begin{equation}
\zeta(s; S) = - (2 \pi i)^{-1}  \ointctrclockwise_{\gamma} dz \, z^{-s} \big[\tr_{\cH}\big((S - z I_{\cH})^{-1}\big) + z^{-1} m(0; S)\big]   \lb{14.2.26}
\end{equation}
for $\Re(s) > 0$ sufficiently large.
\end{lemma}
%%%%%%%
\begin{proof}
Assuming $\Re(s) > 0$ sufficiently large, a contour integration argument yields
\begin{align}
\zeta(s; S) &=  - (2 \pi i)^{-1} \ointctrclockwise_{\gamma} dz \, z^{-s} \tr_{\cH}\big((S - z I_{\cH})^{-1}
[I_{\cH} - P(0; S)]\big)    \no \\
&= - (2 \pi i)^{-1}  \ointctrclockwise_{\gamma} dz \, z^{-s} \tr_{\cH}\big((S - z I_{\cH})^{-1} + z^{-1} P(0; S)\big)  \lb{14.2.9} \\
&= - (2 \pi i)^{-1}  \ointctrclockwise_{\gamma} dz \, z^{-s} \big[\tr_{\cH}\big((S - z I_{\cH})^{-1}\big) + z^{-1} m(0; S)\big],
\no
\end{align}
taking into account that $\tr_{\cH}\big((S - z I_{\cH})^{-1}
[I_{\cH} - P(0; S)]\big)$ is meromorphic with poles precisely at the nonzero eigenvalues $\lambda_j \neq 0$ of $S$, with residues given by (self-adjoint) spectral projections of $S$ of rank equal to $m(\lambda_j; S)$.
\end{proof}
%%%%%%%

It is very tempting to continue the computation leading to \eqref{14.2.26}
and now deform the contour $\gamma$ so as to ``hug'' the branch cut $R_{\theta}$, but this requires the right asymptotic behavior of $\tr_{\cH}\big((S - z I_{\cH})^{-1} [I_{\cH} - P(0; S)]\big)$ as $|z| \to \infty$ as well as $|z| \to 0$, and we will investigate this in the context of relative $\zeta$-functions next.

In cases where $(S - z I_{\cH})^{-1}$ is not trace class, but one is dealing with a pair of operators $(S_1, S_2)$ such that the difference of their resolvents lies in the trace class, \eqref{14.2.9} naturally leads to the notion of a relative $\zeta$-function as follows. For pertinent background information on this circle of ideas we refer, for instance, to Forman \cite{Fo87}, \cite{Fo92}, M\"uller \cite{Mu98}.

%%%%%%%
\begin{hypothesis} \lb{14.h2.7}
Suppose $S_j$, $j=1,2$, are self-adjoint operators in $\cH$, bounded from below, satisfying
\begin{equation}
\big[(S_2- z I_{\cH})^{-1} - (S_1- z I_{\cH})^{-1}\big] \in \cB_1(\cH)
\end{equation}
for some $($and hence for all\,$)$ $z \in \rho(S_1) \cap \rho(S_2)$.
In addition, assume that $S_j$, $j=1,2$, have essential spectrum contained
in $(0,\infty)$, that is, for some $\lambda_1 > 0$,
\begin{equation}
\sigma_{ess}(S_j) \subseteq [\lambda_1, \infty), \quad j=1,2.   \lb{14.2.28}
\end{equation}
\end{hypothesis}
%%%%%%%

We note, in particular, the essential spectrum hypothesis \eqref{14.2.28} includes the case of purely discrete spectra of
$S_j$ (i.e., $\sigma_{ess}(S_j) = \emptyset$, $j=1,2$.) Since $S_j$ were assumed to be bounded from below, adding a suitable constant to $S_j$, $j=1,2$, will shift their (essential) spectra accordingly.

Given Hypothesis \ref{14.h2.7}, and again choosing a counterclockwise oriented simple contour $\gamma$ that encircles
$\sigma(S_1) \cup \sigma(S_2)$, however, with the stipulation
that $0$ does not lie in the interior of $\gamma$, and
$0 \notin \gamma$ (cf.\ Figure \ref{14.2.fig1}), the relative spectral
$\zeta$-function for the pair $(S_1,S_2)$ is defined by
\begin{align}
\begin{split}
& \zeta(s; S_1,S_2) =  - (2 \pi i)^{-1} \ointctrclockwise_{\gamma} dz \, z^{-s}
\tr_{\cH}\big((S_2 - z I_{\cH})^{-1}[I_{\cH} - P(0; S_2)]    \\
& \hspace*{5.7cm} - (S_1 - z I_{\cH})^{-1}[I_{\cH} - P(0; S_1)]\big)   \lb{14.2.29}
\end{split}
\end{align}
for $\Re(s) > 0$ sufficiently large, ensuring convergence of \eqref{14.2.29}.

Employing the contour $\gamma$ and branch cut $R_{\theta}$ as in Figure \ref{14.2.fig1}, and deforming $\gamma$ so it eventually surrounds $R_{\theta}$, one
arrives at the following result.

%%%%%%%
\begin{theorem} \lb{14.t2.8}
Suppose $S_j$, $j=1,2$, are self-adjoint operators in $\cH$ satisfying
Hypothesis \ref{14.h2.7} and that $($cf.\ \eqref{14.2.8B} and
\eqref{14.2.24a}$)$
\begin{align}
& {\tr}_{\cH}\big((S_2 - z I_{\cH})^{-1} - (S_1 - z I_{\cH})^{-1}\big)   \no \\
& \quad = - \f{d}{dz} \ln\Big({\det}_{\cH}\Big(\ol{(S_2 - z I_{\cH})^{1/2}
(S_1 - z I_{\cH})^{-1}
	(S_2 - z I_{\cH})^{1/2}}\Big)\Big),   \lb{14.2.30a} \\
& \hspace*{5.1cm }
 z \in \bbC \backslash [\inf(\sigma(S_1) \cap \sigma(S_2)),\infty).    \no
\end{align}
In addition, assume that for some
$\varepsilon > 0$,
\begin{align}
\begin{split}
& \big|{\tr}_{\cH}\big((S_2- z I_{\cH})^{-1}[I_{\cH} - P(0; S_2)] - (S_1- z I_{\cH})^{-1} [I_{\cH} - P(0; S_1)]\big)\big|    \\
& \quad = \begin{cases} \Oh\big(|z|^{-1 - \varepsilon}\big),
& \text{as $|z| \to \infty$,} \\
\Oh(1),
& \text{as $|z| \to 0$.}
\end{cases}     \lb{14.2.31}
\end{split}
\end{align}
Then, for $\Re(s) \in (- \varepsilon, 1)$,
\begin{align}
& \zeta(s; S_1,S_2) = e^{is (\pi - \theta)} \pi^{-1} \sin(\pi s)
 \int_0^{\infty} dt \, t^{-s} \no \\
& \quad \times \f{d}{dt} \ln\Big((e^{i\theta}t)^{[m(0;S_1) - m(0;S_2)]}   \lb{14.2.32} \\
& \quad \times {\det}_{\cH}\big((S_2 - e^{i \theta} t I_{\cH})^{1/2}
(S_1 - e^{i \theta} t I_{\cH})^{-1} (S_2 - e^{i \theta} t I_{\cH})^{1/2}\big)\Big).   \no
\end{align}
\end{theorem}
%%%%%%%%
\begin{proof}
Due to hypothesis \eqref{14.2.31} one can deform the contour $\gamma$
so that it wraps around the branch cut $R_{\theta}$,
\begin{align}
& \zeta(s; S_1,S_2) =  - (2 \pi i)^{-1} \ointctrclockwise_{\gamma} dz \, z^{-s}
\tr_{\cH}\big((S_2 - z I_{\cH})^{-1}[I_{\cH} - P(0; S_2)]   \no \\
& \hspace*{5.8cm} - (S_1 - z I_{\cH})^{-1}[I_{\cH} - P(0; S_1)]\big)  \no \\
& \quad =- (2 \pi i)^{-1} \ointctrclockwise_{\gamma} dz \, z^{-s}
\big\{\tr_{\cH}\big((S_2 - z I_{\cH})^{-1} - (S_1 - z I_{\cH})^{-1}\big)    \no\\
&\hspace*{3.8cm} + z^{-1} \left[ m(0;S_2) - m(0;S_1)\right] \big\}\no \\
& \quad = (2\pi i)^{-1} \ointctrclockwise_{\gamma} dz \, z^{-s}
\bigg\{ \frac d {dz} \ln \Big({\det}_{\cH} \Big( ( S_2 - z I_{\cH})^{1/2} (S_1 - z I_{\cH})^{-1}
( S_2 - z I_{\cH})^{1/2}\Big)\Big)     \no\\
& \hspace*{3.7cm} - z^{-1}  [m (0;S_2) - m(0;S_1)] \bigg\}\no\\
& \quad = (2\pi i)^{-1} \ointctrclockwise_{\gamma} dz \, z^{-s}
\frac d {dz} \ln \Big( z^{[m(0;S_1)-m(0;S_2)]}     \no\\
& \hspace*{4.2cm} \times {\det}_{\cH} \Big(\left( S_2 - z I_{\cH}\right)^{1/2}  \left(S_1 - z I_{\cH}\right)^{-1} \left( S_2 - z I_{\cH}\right)^{1/2}\Big)\Big)      \no\\
& \quad = e^{is (\pi -\theta )} \pi^{-1} \sin (\pi s) \int_0^\infty dt \, t^{-s} \frac d {dt} \ln \Big( \left(t e^{i\theta}\right)^{[m(0;S_1) - m(0;S_2)]} \no\\
& \qquad \times {\det}_{\cH} \Big(\left( S_2 - e^{i\theta} t I_{\cH}\right)^{1/2}  \left(S_1 - e^{i\theta} t I_{\cH}\right)^{-1} \left( S_2 - e^{i\theta} t I_{\cH}\right)^{1/2}\Big)\Big).
% \lb{14.2.35}
\end{align}
Here we first applied $\tr_{\cH}(P(0;S_j)) = m(0;S_j)$ (cf.\ also \eqref{14.2.29}) and then \eqref{14.2.8A}. Carefully paying attention to the phases when shrinking
the contour to the branch cut $R_\theta$, one obtains \eqref{14.2.32}.
\end{proof}
%%%%%%%

%%%%%%%
\begin{theorem} \lb{14.t2.10}
Suppose $S_j$, $j=1,2,$ are self-adjoint operators in $\cH$ satisfying Hypothesis \ref{14.h2.7} and
\eqref{14.2.31}. Then
\begin{align}
\begin{split}
& \zeta'(0; S_1,S_2) = i\pi (n_2 - n_1) - \lim_{t \downarrow 0}
\left|\ln\Big((e^{i \theta} t)^{[m(0;S_1) - m(0; S_2)]} \right.  \\
& \quad  \left. \times {\det}_{\cH} \big((S_2 - e^{i \theta} t I_{\cH})^{1/2}
(S_1 - e^{i \theta} t I_{\cH})^{-1}
(S_2 - e^{i \theta} t I_{\cH})^{1/2}\big)\Big)  \right| , \lb{14.2.33}
\end{split}
\end{align}
where $n_j$ is the number of negative eigenvalues of $S_j$, $j=1,2$. If $n_j=0$, $j=1,2,$ then
\begin{align}
\zeta'(0; S_1,S_2) &=  - \lim_{t \downarrow 0}
\ln\Big((e^{i \theta} t)^{[m(0;S_1) - m(0; S_2)]} \no \\
& \hspace*{1.7cm}  \times {\det}_{\cH} \big((S_2 - e^{i \theta} t I_{\cH})^{1/2}
(S_1 - e^{i \theta} t I_{\cH})^{-1}
(S_2 - e^{i \theta} t I_{\cH})^{1/2}\big)\Big)  \no \\
& =  - \ln (C_0),  \lb{14.2.33a}
\end{align}
where
\begin{align}
\begin{split}
& {\det}_{\cH}\big((S_2 - z I_{\cH})^{1/2} (S_1 - z I_{\cH})^{-1}
(S_2 - z I_{\cH})^{1/2}\big)   \\
& \quad \underset{z \to 0}{=} z^{[m(0;S_2) - m(0;S_1)]}
[C_0 + \Oh(z)], \quad C_0 > 0.    \lb{14.2.34}
\end{split}
\end{align}
\end{theorem}
%%%%%%%
\begin{proof}
First we note that \eqref{14.2.32} implies
\begin{align}
\begin{split}
& \zeta ' (0; S_1,S_2) = \int_0^\infty dt \,\,
 \frac d {dt} \ln \Big\{ \left(t e^{i\theta}\right)^{[m(0;S_1) - m(0;S_2)]}    \\
& \quad \times {\det}_{\cH} \Big(\left( S_2 - e^{i\theta} t I_{\cH}\right)^{1/2}
\left(S_1 - e^{i\theta} t I_{\cH}\right)^{-1} \left( S_2 - e^{i\theta} t I_{\cH}\right)^{1/2}\Big)\Big\}.   \lb{14.2.41}
\end{split}
\end{align}
In computing this quantity, one notes that for $t\in [0,\infty )$, the graph of the function
\begin{align}
\begin{split}
G(t) &= \left(t e^{i\theta}\right)^{[m(0;S_1) - m(0;S_2)]}     \\
& \quad \times {\det}_{\cH} \Big(\left( S_2 - e^{i\theta} t I_{\cH}\right)^{1/2}
\left(S_1 - e^{i\theta} t I_{\cH}\right)^{-1} \left( S_2 - e^{i\theta} t I_{\cH}\right)^{1/2}\Big)
\end{split}
\end{align}
can cross the branch cut $R_\theta$ at several $t$-values. Therefore, the
integral has to be split at these $t$-values and pursuant real and imaginary parts have to be summed. The real part
between consecutive segments cancels except for the contributions from zero. This explains the real part of \eqref{14.2.33}.
The resulting imaginary part is found as follows. The sum defining the $\zeta$-function can be split into contributions from
negative and positive eigenvalues, namely,
\begin{equation}
\zeta (s; S_j ) = \sum_{\substack{ \ell =1\\ \lambda_{-\ell} < 0 }}^{n_j} \lambda_{-\ell}^{-s} + \sum_{\substack{k\in J \\ \lambda_k > 0}} \lambda_k^{-s}.
\end{equation}
For each negative eigenvalue one computes
\begin{equation}
\left. \frac d {ds} \right|_{s=0} \lambda_{-\ell}^{-s} = - \ln (\lambda _{-\ell}) = i \pi - \ln (| \lambda_{-\ell } |).
\end{equation}
This yields the imaginary part in \eqref{14.2.33}.

Since $0$ lies
outside the essential spectra of $S_j$, $j=1,2$, \eqref{14.2.34} follows, for instance, from \cite[p.~271--272]{Ya92}. Given the relation \eqref{14.2.34},
the fact \eqref{14.2.33a} follows from \eqref{14.2.41} and the Lebesgue dominated convergence theorem.
\end{proof}
%%%%%%%

In the special case where $0 \in \rho(S_1) \cap \rho(S_2)$ (i.e.,
$m(0;S_1) = m(0;S_2) = 0$), one thus obtains
\begin{equation}
e^{- \zeta'(0; S_1,S_2)} = {\det}_{\cH} \big(S_2^{1/2} S_1^{-1} S_2^{1/2}\big)
\end{equation}
and hence the $\zeta$-function regularized relative determinant now
equals the symmetrized (Fredholm) perturbation determinant for the
pair $(S_1,S_2)$. Here any choice of branch cut of the self-adjoint operator $S_2^{1/2}$ (employing the spectral theorem) is permissible.

%%%%%%%%%%%%%%%%%%%%%%%%%
%%%%%%%%%%%%%%%%%%%%%%%%%
\section{Sturm--Liouville Operators on Bounded Intervals} \lb{14.s3}
%%%%%%%%%%%%%%%%%%%%%%%%%
%%%%%%%%%%%%%%%%%%%%%%%%%

To illustrate the material of Section \ref{14.s2} we now apply it to the case of self-adjoint Sturm--Liouville operators on bounded intervals.

We start by recalling a convenient parametrization of all self-adjoint
extensions associated with a regular, symmetric, second-order differential expression
as discussed in detail, for instance, in \cite[Theorem\ 13.15]{We03} and \cite[Theorem\ 10.4.3]{Ze05}, and recorded in \cite{CGNZ14}.

Throughout this section we make the following set of assumptions:

%%%%%%%%%%%%
\begin{hypothesis} \lb{14.h3.1}
Suppose $p, q, r$ satisfy the following conditions: \\
$(i)$ \,\, $r>0$ a.e.\ on $(a,b)$, $r\in L^1((a,b); dx)$. \\
$(ii)$ \, $p>0$  a.e.\ on $(a,b)$, $1/p\in L^1((a,b); dx)$. \\
$(iii)$ $q\in L^1((a,b); dx)$, $q$ is real-valued a.e.\ on $(a,b)$.
\end{hypothesis}
%%%%%%%%%%%%

Given Hypothesis \ref{14.h3.1}, we take $\tau$ to be the Sturm--Liouville-type differential expression
defined by
\begin{equation} \lb{14.3.1}
\tau=\frac{1}{r(x)}\left[-\frac{d}{dx}p(x)\frac{d}{dx} + q(x)\right]
 \, \text{ for a.e.~$x\in(a,b)$,} \quad -\infty < a < b < \infty,
\end{equation}
and note that $\tau$ is regular on $[a,b]$. In addition, the following convenient notation for the
\emph{first quasi-derivative} is introduced,
\begin{equation}
y^{[1]}(x)=p(x)y'(x) \, \text{ for a.e.\ $x\in(a,b)$, } \, y \in AC([a,b]).     \lb{14.3.2}
\end{equation}
Here $AC([a,b])$ denotes the set of absolutely continuous functions on $[a,b]$. For notational convenience we will occasionally abbreviate $L^2_r((a,b)) := L^2((a,b); rdx)$.

Given that $\tau$ is regular on $[a,b]$, the \emph{maximal operator}
$H_{max}$ in $L^2((a,b);rdx)$ associated with $\tau$ is defined by
\begin{align}
&H_{max} f=\tau f,     \lb{14.3.3} \\
& \, f \in \dom(H_{max})= \big\{g\in L^2((a,b);rdx) \, \big| \, g,  g^{[1]}\in AC([a,b]); \,
\tau g\in  L^2((a,b);rdx)\big\},   \no
\intertext{while the \emph{minimal operator} $H_{min}$  in $L^2((a,b);rdx)$ associated with
$\tau$ is given by}
&H_{min} f=\tau f,    \no \\
& \, f \in \dom(H_{min})= \big\{g\in L^2((a,b);rdx) \, \big| \, g,  g^{[1]}\in AC([a,b]);    \lb{14.3.4} \\
&\hspace*{3.1cm}g(a)=g^{[1]}(a)=g(b)=g^{[1]}(b)=0; \, \tau g\in L^2((a,b);rdx)\big\}.      \no
\end{align}

One notes that the operator $H_{min}$ is symmetric and that
\begin{equation}
H_{min}^*=H_{max}, \quad H_{max}^*=H_{min},     \lb{14.3.5}
\end{equation}
(cf. Weidmann \cite[Theorem 13.8]{We03}). Next, we summarize material found, for instance, in
\cite[Ch.\ 13]{We03} and \cite[Sects.\ 10.3, 10.4]{Ze05} in which self-adjoint extensions of the minimal operator $H_{min}$ are characterized.

%%%%%%%%%%%%
\begin{theorem} $($See, \cite[ Satz~2.6]{We75}, \cite[Theorem~13.14]{We03},
\cite[Theorem~10.4.2]{Ze05}.$)$  \lb{14.t3.2}
Assume Hypothesis \ref{14.h3.1} and suppose that $\widetilde H$  is an  extension of the minimal operator $H_{min}$ defined in \eqref{14.3.4}. Then the following hold: \\
$(i)$ $\widetilde H$ is a self-adjoint extension of $H_{min}$ if and only if there exist $2\times 2$ matrices $A$ and $B$ with complex-valued entries satisfying
\begin{equation} \lb{14.3.6}
\rank (A\ \ B)=2, \quad AJA^*=BJB^*, \quad
J=\begin{pmatrix} 0 & -1 \\ 1 & 0\end{pmatrix},
\end{equation}
with
\begin{align} \lb{14.3.7}
\begin{split}
& \widetilde Hf= \tau f   \\
& \dom\big(\widetilde H\big)=\left\{ g\in\dom(H_{max})\, \bigg| \, A\binom{g(a)}{g^{[1]}(a)} =
B\binom{g(b)}{g^{[1]}(b)}\right\}.
\end{split}
\end{align}
Henceforth, the self-adjoint extension $\widetilde H$ corresponding to the matrices
 $A$ and $B$ will  be denoted by $H_{A,B}$. \\
$(ii)$ For $z\in\rho(H_{A,B})$, the resolvent $H_{A,B}$ is of the form
\begin{equation}
\begin{split}
\left((H_{A,B} - z I_{L_r^2((a,b))})^{-1}f\right)(x)= \int_a^br(x')dx' \, G_{A,B}(z,x,x')f(x'),    \\
f\in L^2((a,b);rdx),   \lb{14.3.8}
\end{split}
\end{equation}
where the Green's function $G_{A,B}(z,x,x')$ is of the form given by
\begin{equation}
\hspace*{10pt}G_{A,B}(z,x,x')= \begin{cases}
\sum_{j,k=1}^{2} m_{j,k}^+(z)u_j(z,x)u_k(z,x'), & a \leq x'\le x \le b,  \\[2pt]
\sum_{j,k=1}^{2} m_{j,k}^-(z)u_j(z,x)u_k(z,x'), & a \leq x < x' \le b.     \lb{14.3.9}
\end{cases}
\end{equation}
Here $\{u_1 (z,\, \cdot \,), u_2(z,\,\cdot\,)\}$ represents a fundamental set of solutions for
$(\tau -z)u=0$ and
$m_{j,k}^\pm(z)$, $1 \leq j, k \leq 2$, are appropriate constants. \\
$(iii)$ $H_{A,B}$ has purely discrete spectrum with eigenvalues of multiplicity
at most $2$. Moreover, if $\sigma(H_{A,B})=\{\lambda_{A,B,j}\}_{j\in\bbN}$,
then
\begin{equation}
\sum_{\substack{j\in\bbN\\ \lambda_{A,B,j}\ne 0}}|\lambda_{A,B,j}|^{-1}<\infty.   \lb{14.3.11}
\end{equation}
In particular,
\begin{equation}
(H_{A,B} - z I_{L_r^2((a,b))})^{-1} \in \cB_1\big(L^2((a,b); r dx)\big), \quad z \in \rho(H_{A,B}).   \lb{14.3.10}
\end{equation}
\end{theorem}
%%%%%%%%%%

The characterization of self-adjoint extensions of $H_{\text{min}}$ in terms of pairs of matrices $(A,B)\in \bbC^{2\times 2}\times \bbC^{2\times 2}$ satisfying \eqref{14.3.6} is not unique
in the sense that different pairs may lead to the same self-adjoint extension. The next result recalls a unique characterization
for all self-adjoint extensions of $H_{min}$ and hence can be viewed as a refinement of Theorem \ref{14.t3.2}.

%%%%%%%%%%%
\begin{theorem} $($See, e.g., \cite[Theorem~13.15]{We03}, \cite[, Theorem~10.4.3]{Ze05}.$)$ \lb{14.t3.3}
Assume Hypothesis \ref{14.h3.1}.  Let $H_{min}$ be the minimal operator associated with
$\tau$ and defined in \eqref{14.3.4} and $H_{A,B}$ a self-adjoint extension of the minimal operator as characterized in Theorem\ \ref{14.t3.2}; then, the following hold: \\
$(i)$ $H_{A,B}$ is a self-adjoint extension of $H_{min}$, with $\rank(A)=\rank(B)=1$ if and only if
$A$ and $B$ can be put in the form
\begin{equation} \lb{14.3.12}
A=\begin{pmatrix}\cos(\alpha)&\sin(\alpha)\\0&0 \end{pmatrix},\quad
B=\begin{pmatrix}0&0\\ -\cos(\beta)&\sin(\beta) \end{pmatrix},
\end{equation}
for a unique pair $\alpha, \beta\in[0,\pi)$. Hence, upon identifying $H_{A,B}$ with $H_{\alpha,\beta}$,
all self-adjoint extensions of $H_{min}$ with separated boundary conditions are of the form
\begin{align} \lb{14.3.13}
& H_{\alpha,\beta} f = \tau f, \quad \alpha, \beta\in[0,\pi),     \no \\
&f \in \dom(H_{\alpha,\beta})=\{ g\in\dom(H_{max}) \, | \,
\sin(\alpha) g^{[1]}(a) + \cos(\alpha) g(a) = 0,    \\
& \hspace*{5.35cm} - \sin(\beta) g^{[1]}(b) + \cos(\beta)  g(b) = 0 \}.   \no
\end{align}
$(ii)$ $H_{A,B}$ is a self-adjoint extension of $H_{min}$  with
$\rank(A)=\rank(B)=2$ if and only if $A$ and $B$ can be put in the form
\begin{equation} \lb{14.3.14}
A=e^{i\varphi} R,\quad B=I_2,
\end{equation}
for a unique $\varphi\in[0,2\pi)$, and unique $R \in \SL_2(\mathbb{R})$.
Hence, upon identifying $H_{A,B}$ with $H_{\varphi,R}$,
all self-adjoint extensions of $H_{min}$ with coupled boundary conditions are of the form
\begin{align} \lb{14.3.15}
\begin{split}
& H_{\varphi,R} f = \tau f, \quad \varphi\in[0,2\pi), \; R \in \SL_2(\mathbb{R}),   \\
& f \in \dom(H_{\varphi,R})=\left\{g\in\dom(H_{max})
 \, \bigg| \, \binom{g(b)}{g^{[1]}(b)}=e^{i\varphi} R
\binom{g(a)}{g^{[1]}(a)} \right\}.
\end{split}
\end{align}
$(iii)$ All self-adjoint extensions of $H_{min}$ are either of type $(i)$ $($i.e., separated\,$)$ or of type
$(ii)$ $($i.e., coupled\,$)$.
\end{theorem}
%%%%%%%%%%%%

Here $\SL_2(\mathbb{R})$ denotes the group of $2 \times 2$ matrices with real-valued entries and determinant one. 

For notational convenience we will adhere to the notation $H_{\alpha,\beta}$ and $H_{\varphi,R}$
in the following.

Next, we recall some results of \cite{GW95}. For this purpose we introduce the fundamental system
of solutions $\theta(z,x,a)$, $\phi(z,x,a)$ of $\tau y =z y$ defined by
\begin{equation}
\theta(z,a,a)=\phi^{[1]}(z,a,a)=1, \quad
\theta^{[1]}(z,a,a)=\phi(z,a,a)=0,    \lb{14.3.16}
\end{equation}
such that
\begin{equation}
W(\theta(z, \, \cdot \,,a), \phi(z,\, \cdot \,,a)) = 1,    \lb{14.3.16a}
\end{equation}
where, for $f,g$ (locally) absolutely continuous,
\begin{equation}
W(f,g)(\, \cdot \,) = f(\, \cdot \,) g^{[1]}(\, \cdot \,) - f^{[1]}(\, \cdot \,) g(\, \cdot \,).    \lb{14.3.17a}
\end{equation}
Furthermore, we introduce
the boundary values for $g, \, g^{[1]}\in AC([a,b])$, see \cite[Sect.~1.2]{Na67},
\cite[Sect.~3.2]{Ze05},
\begin{align}
\begin{split}
U_{\alpha,\beta,1} (g) &= \sin(\alpha) g^{[1]}(a) + \cos(\alpha) g(a), \\
U_{\alpha,\beta,2} (g) &= - \sin(\beta) g^{[1]}(b) + \cos(\beta) g(b),   \lb{14.3.17}
\end{split}
\end{align}
in the case $(i)$ of separated boundary conditions in Theorem \ref{14.t3.3}, and
\begin{align}
\begin{split}
V_{\varphi,R,1}(g) &= g(b) - e^{i \varphi} R_{1,1} g(a) - e^{i \varphi} R_{1,2} g^{[1]} (a), \\
V_{\varphi,R,2}(g) &= g^{[1]}(b) - e^{i \varphi} R_{2,1} g(a) - e^{i \varphi} R_{2,2} g^{[1]} (a),  \lb{14.3.18}
\end{split}
\end{align}
in the case $(ii)$ of coupled boundary conditions in Theorem \ref{14.t3.3}. Moreover, we define
\begin{align}
F_{\alpha,\beta}(z) = {\det}_{\bbC^2}\begin{pmatrix} U_{\alpha,\beta,1}(\theta(z, \, \cdot \,,a))
& U_{\alpha,\beta,1}(\phi(z, \, \cdot \,,a)) \\[1mm]
U_{\alpha,\beta,2}(\theta(z, \, \cdot \,,a)) & U_{\alpha,\beta,2}(\phi(z, \, \cdot \,,a)) \end{pmatrix},
\lb{14.3.19}
\end{align}
and
\begin{align}
F_{\varphi,R}(z) = {\det}_{\bbC^2}\begin{pmatrix} V_{\varphi,R,1}(\theta(z, \, \cdot \,,a))
& V_{\varphi,R,1}(\phi(z, \, \cdot \,,a)) \\[1mm]
V_{\varphi,R,2}(\theta(z, \, \cdot \,,a)) & V_{\varphi,R,2}(\phi(z, \, \cdot \,,a)) \end{pmatrix},   \lb{14.3.20}
\end{align}
and note, in particular, that 
\begin{align}
& F_{\alpha,\beta}(z)  \lb{14.3.21} \\
& \quad =\begin{cases}
\phi(z,b,a), & \alpha = \beta = 0, \\[1mm]
- \sin(\beta) \phi^{[1]}(z,b,a) + \cos(\beta) \phi(z,b,a),  & \alpha = 0, \, \beta \in (0,\pi),
\\[1mm]
\cos(\alpha) \phi(z,b,a) - \sin(\alpha) \theta(z,b,a), & \alpha \in (0,\pi), \, \beta = 0,
\\[1mm]
\cos(\alpha) [- \sin(\beta) \phi^{[1]}(z,b,a) + \cos(\beta) \phi(z,b,a)] \\[.5mm]
- \sin(\alpha) [- \sin(\beta) \theta^{[1]}(z,b,a) + \cos(\beta) \theta(z,b,a)], & \alpha, \beta \in (0,\pi).
\end{cases}   \no
\end{align}

Given these preparations we can state our first result concerning the computation of traces and determinants.

%%%%%%%
\begin{theorem} \lb{14.t3.4}
Assume Hypothesis \ref{14.h3.1} and denote by $H_{\alpha,\beta}$ and $H_{\varphi,R}$ the self-adjoint extensions of $H_{min}$ as described in cases $(i)$ and $(ii)$ of Theorem \ref{14.t3.3}, respectively. \\
$(i)$ Suppose $z_0 \in \rho(H_{\alpha,\beta})$, then
\begin{align}
\begin{split}
& {\det}_{L^2_r((a,b))} \big(I_{L^2_r((a,b))} - (z - z_0) (H_{\alpha,\beta} - z_0 I_{L^2_r((a,b))})^{-1}\big)
\\
& \quad = F_{\alpha,\beta}(z)/F_{\alpha,\beta}(z_0), \quad z \in \bbC.    \lb{14.3.22}
\end{split}
\end{align}
In particular,
\begin{equation}
\tr_{L^2_r((a,b))} \big((H_{\alpha,\beta} - z I_{L^2_r((a,b))})^{-1}\big) =
- (d/dz) \ln(F_{\alpha,\beta}(z)), \quad z \in \rho(H_{\alpha,\beta}).     \lb{14.3.23}
\end{equation}
$(ii)$ Suppose $z_0 \in \rho(H_{\varphi,R})$, then
\begin{align}
\begin{split}
& {\det}_{L^2_r((a,b))} \big(I_{L^2_r((a,b))} - (z - z_0) (H_{\varphi,R} - z_0 I_{L^2_r((a,b))})^{-1}\big)
\\
& \quad = F_{\varphi,R}(z)/F_{\varphi,R}(z_0), \quad z \in \bbC.     \lb{14.3.24}
\end{split}
\end{align}
In particular,
\begin{equation}
\tr_{L^2_r((a,b))} \big((H_{\varphi,R} - z I_{L^2_r((a,b))})^{-1}\big) =
- (d/dz) \ln(F_{\varphi,R}(z)), \quad z \in \rho(H_{\varphi,R}).       \lb{14.3.25}
\end{equation}
\end{theorem}
%%%%%%%
\begin{proof}
In the special case $p = r =1$ and given separated (but generally, non-self-adjoint) boundary conditions, the fact \eqref{14.3.22} was proved in \cite{GW95} upon combining the eigenvalue results in
\cite[Sect.~1.2]{Na67} and Lemma \ref{14.l2.2}. The proof in \cite{GW95} extends to the present situation with $p, r$ satisfying Hypothesis \ref{14.h3.1} in the cases $(i)$ and $(ii)$ since actually the eigenvalues
$\lambda_{A,B,j}$ of $H_{A,B}$ all have the universal leading Weyl-type asymptotics
\cite[Sect.~VI.7]{GK70} (see also \cite{AM87}, \cite[Theorem~4.3.1]{Ze05})
\begin{equation}
\lim_{j \to \infty} j^{-2} \lambda_{A,B,j} = \pi^2 \bigg(\int_a^b dx \, [r(x)/p(x)]^{1/2}\bigg)^{-2},  \lb{14.3.26}
\end{equation}
independently of the chosen boundary conditions, improving upon relation \eqref{14.3.11}. More precisely, \cite[Sect.~VI.7]{GK70} determines the leading asymptotic behavior of the eigenvalue counting function (independently of the underlying choice of boundary conditions) which is known to yield \eqref{14.3.26}.

Relations \eqref{14.3.23} and \eqref{14.3.25} are then clear from combining formula \eqref{14.2.3} with \eqref{14.3.22} and \eqref{14.3.24}, respectively.
\end{proof}
%%%%%%%

%%%%%%%
\begin{remark} \lb{14.r3.18}
Considering traces of resolvent differences for various boundary conditions
(separated and/or coupled) permits one to make a direct connection with the boundary data maps studied in \cite{CGM10}, \cite{CGNZ14}, and \cite{GZ12}. More precisely, suppose the pairs $A, B \in \bbC^{2 \times 2}$ and $A', B' \in \bbC^{2 \times 2}$ satisfy \eqref{14.3.6} and define
$H_{A,B}$ and $H_{A',B'}$ according to \eqref{14.3.7}. Then
\begin{align}
&  \tr_{L^2((a,b); rdx)}\big((H_{A',B'} -z I_{L_r^2((a,b))})^{-1}
- (H_{A,B} -z I_{L_r^2((a,b))})^{-1}\big)    \no \\
& \quad = - \f{d}{dz} \ln\Big({\det}_{\bbC^2}\Big(\Lambda_{A,B}^{A',B'} (z)\Big)\Big), \quad
z \in \rho(H_{A,B}) \cap \rho(H_{A',B'}),   \lb{14.3.44}
\end{align}
where the $2 \times 2$ matrix $\Lambda_{A,B}^{A',B'} (z)$ represents
the boundary data map associated to the pair $(H_{A,B}, H_{A',B'})$. A comparison of \eqref{14.3.44} with \eqref{14.3.23} and \eqref{14.3.25} yields
\begin{equation}
{\det}_{\bbC^2}\Big(\Lambda_{A,B}^{A',B'} (z)\Big)
= C_0 \, F_{A',B'}(z)/F_{A,B}(z),
\quad z \in \rho(H_{A,B}) \cap \rho(H_{A',B'}),   \lb{14.3.45}
\end{equation}
where, in obvious notation, $F_{A',B'}(z), F_{A,B}(z)$ represent either
$F_{\alpha,\beta}(z)$ or $F_{\varphi,R}(z)$, depending on which boundary
conditions (separated or coupled) are represented by the pairs
$(A',B')$, $(A,B)$,
and $C_0 = C_0 (A',B',A,B)$ is a $z$-independent constant.
Indeed, $C_0=1$ for separated as well as coupled boundary conditions. The separated case was shown in
\cite{GZ12}, Lemma 3.4, using the explicit matrix representation of $\Lambda_{AB}^{A'B'}(z)$
for that case, and the result for coupled boundary conditions follows along the same lines.
\hfill $\diamond$
\end{remark}
%%%%%%%

For the case of determinants of general higher-order differential operators (with matrix-valued coefficients)  and general boundary conditions on bounded intervals we refer to Burghelea, Friedlander, and Kappeler
\cite{BFK95}, Dreyfus and Dym \cite{DD78}, Falco, Fedorenko, and Gruzberg \cite{FFG17},
and Forman \cite[Sect.~3]{Fo87}. We also refer to \cite{LS77} for a closed form of an infinite
product of ratios of eigenvalues of Sturm--Liouville operators on bounded intervals.

We briefly look at two prominent examples associated with coupled boundary conditions next:

%%%%%%%
\begin{example} $($Floquet boundary conditions.\,$)$ \lb{14.e3.5}
Consider the family of operators $H_{\varphi,I_2}$, familiar from Floquet theory, defined by taking $R = B = I_2$ in \eqref{14.3.15}. In this case
\begin{equation}
F_{\varphi,I_2}(z) = - 2 e^{i \varphi} [\Delta(z) - \cos(\varphi)], \quad z \in \bbC,   \lb{14.3.27}
\end{equation}
where
\begin{equation}
\Delta(z) = \big[\theta(z,b,a) + \phi^{[1]}(z,b,a)\big]\big/2, \quad z \in \bbC,    \lb{14.3.28}
\end{equation}
represents the well-known Floquet discriminant and hence
\begin{align}
& {\det}_{L^2_r((a,b))} \big(I_{L^2_r((a,b))} - (z - z_0) (H_{\varphi,I_2} - z_0 I_{L^2_r((a,b))})^{-1}\big)
\no \\
& \quad = [\Delta(z) - \cos(\varphi)]/[\Delta(z_0) - \cos(\varphi)], \quad z \in \bbC,     \lb{14.3.29} \\
& \tr_{L^2_r((a,b))} \big((H_{\varphi,I_2} - z I_{L^2_r((a,b))})^{-1}\big)   \no \\
& \quad = - (d/dz) \ln([\Delta(z) - \cos(\varphi)]), \quad z \in \rho(H_{\varphi,I_2}).       \lb{14.3.30}
\end{align}
\end{example}
%%%%%%%

%%%%%%%
\begin{example} $($The Krein--von Neumann extension.\,$)$ \lb{14.e3.6}
Consider the case $\varphi =0$, $A = R_K$, with
\begin{equation}
R_K = \begin{pmatrix} \theta(0,b,a) & \phi(0,b,a) \\[1mm]
\theta^{[1]}(0,b,a) & \phi^{[1]}(0,b,a) \end{pmatrix}.    \lb{14.3.31}
\end{equation}
As shown in \cite[Example~3.3]{CGNZ14}, the resulting operator $H_{0,R_K}$ represents the
Krein--von Neumann extension of $H_{min}$. Thus,
\begin{equation}
F_{0,R_K}(z) =  - 2 [D_K(z) - 1], \quad z \in \bbC,   \lb{14.3.32}
\end{equation}
where
\begin{align}
\begin{split}
D_K(z) &= \big[ \phi^{[1]}(0,b,a) \theta(z,b,a) + \theta(0,b,a) \phi^{[1]}(z,b,a) - \phi(0,b,a) \theta^{[1]}(z,b,a) \\
& \quad - \theta^{[1]}(0,b,a) \phi(z,b,a)\big]\big/2, \quad z \in \bbC.    \lb{14.3.33}
\end{split}
\end{align}
Hence,
\begin{align}
& {\det}_{L^2_r((a,b))} \big(I_{L^2_r((a,b))} - (z - z_0) (H_{0,R_K} - z_0 I_{L^2_r((a,b))})^{-1}\big)
\no \\
& \quad = [D_K(z) - 1]/[D_K(z_0) - 1], \quad z \in \bbC,     \lb{14.3.34} \\
& \tr_{L^2_r((a,b))} \big((H_{0,R_K} - z I_{L^2_r((a,b))})^{-1}\big)    \no \\
& \quad = - (d/dz) \ln(D_K(z) - 1), \quad z \in \rho(H_{0,R_K}).       \lb{14.3.35}
\end{align}
Because of the Wronskian relation \eqref{14.3.16a}, $D_K(0) = 1$, furthermore
\begin{equation}
D_K(z) - 1 \underset{z \to 0}{=} z^2 [c + \Oh(z)],  \lb{14.3.37}
\end{equation}
where
\begin{align}
\begin{split}
c &= \frac{1}{2} \Big(\dot \phi (0,b,a) \, \big[\dot \theta \big]^{[1]} (0,b,a) -
\big[\dot \phi\big]^{[1]} (0,b,a) \, \dot \theta (0,b,a)\Big)     \lb{14.3.38} \\
&= \frac{1}{2} W\Big(\dot \phi (0,\, \cdot \,,a), \dot \theta (0,\, \cdot \,,a) \Big)(b),
\end{split}
\end{align}
abbreviating $\dot{} = d/dz$.

The small-$z$ behavior \eqref{14.3.37} is clear from general results in \cite[Sect.~1.2]{Na67} or
\cite[Sect.~3.2]{Ze05} and is in accordance with a two-dimensional nullspace of the Krein--von Neumann extension on a bounded interval. To actually show that $c \neq 0$ $($in fact, $c < 0$$)$ in \eqref{14.3.37} one recalls that
$\dot \phi (z,x,a)$ and $\dot \theta (z,x,a)$ satisfy an inhomogeneous Sturm--Liouville equation and as a
result one obtains $($\,for $z \in \bbC$, $x \in [a,b]$$)$
\begin{align}
\begin{split}
\dot \theta (z,x,a) &= \theta (z,x,a) \int_a^x r(x') dx' \, \phi (z,x',a) \theta (z,x',a)      \\
& \quad - \phi (z,x,a) \int_a^x r(x') dx' \, \theta (z,x',a)^2,   \\
\big[\dot \theta \big]^{[1]} (z,x,a) &= \theta^{[1]} (z,x,a) \int_a^x r(x') dx' \, \phi (z,x',a) \theta (z,x',a)     \\
& \quad - \phi^{[1]} (z,x,a) \int_a^x r(x') dx' \, \theta (z,x',a)^2,   \\
\dot \phi (z,x,a) &= \theta (z,x,a) \int_a^x r(x') dx' \, \phi (z,x',a)^2    \lb{14.3.39} \\
& \quad - \phi (z,x,a) \int_a^x r(x') dx' \, \phi (z,x',a) \theta (z,x',a),   \\
\big[\dot \phi \big]^{[1]} (z,x,a) &= \theta^{[1]} (z,x,a) \int_a^x r(x') dx' \, \phi (z,x',a)^2     \\
& \quad - \phi^{[1]} (z,x,a) \int_a^x r(x') dx' \, \phi (z,x',a) \theta (z,x',a).
\end{split}
\end{align}
Equations  \eqref{14.3.39} and Cauchy's inequality then imply
\begin{align}
& W\big(\dot \phi (\lambda,\, \cdot \,,a), \dot \theta (\lambda,\, \cdot \,,a) \big)(x) =
\bigg[\int_a^x r(x') dx' \,  \phi (\lambda ,x',a) \theta (\lambda ,x',a) \bigg]^2    \\
& \quad - \bigg[\int_a^x r(x') dx' \, \theta (\lambda,x',a)^2\bigg] \bigg[\int_a^x r(x') dx' \,
\phi (\lambda,x',a)^2\bigg] \leq 0, \quad \lambda \in \bbR, \; x \in [a,b],    \no
\end{align}
using the fact that $\phi (\lambda,x,a)$ and $\theta (\lambda,x,a)$ are real-valued for
$\lambda \in \bbR$, $x \in [a,b]$. Equality in Cauchy's inequality for $x > a$ would imply that for
some $\alpha, \beta \in [0,\infty)$, $(\alpha, \beta) \neq (0,0)$,
\begin{equation}
\alpha \, \phi (\lambda,x',a) = \beta \, \theta (\lambda,x',a), \quad x' \in (a,x],
\end{equation}
a contradiction. Thus,
$W\big(\dot \phi (\lambda,\, \cdot \,,a), \dot \theta (\lambda,\, \cdot \,,a) \big)(x) < 0$ for
$\lambda \in \bbR$, $x \in (a, b]$ and hence $c < 0$ in \eqref{14.3.37}, \eqref{14.3.38}.
\end{example}
%%%%%%%

In the case of separated boundary conditions, that is, case $(i)$ in
Theorem \ref{14.t3.3}, one can shed more light on
$F_{\alpha, \beta}(z)$ in terms of appropriate Weyl solutions
$\psi_-(z, \, \cdot \,,a,\alpha)$ and $\psi_+(z, \, \cdot \,,a,\beta)$ that satisfy the boundary conditions in $\dom(H_{\alpha,\beta})$ in \eqref{14.3.13} at $a$ and $b$, respectively. Up to normalizations,
$\psi_\pm(z, \, \cdot \, ,a,\alpha )$ are given by
\begin{align}
& \psi_-(z, \, \cdot \,,a,\alpha)
= c_- \begin{cases} \cos(\alpha) \phi(z,\,\cdot\,,a)
- \sin(\alpha) \theta(z,\,\cdot\,,a), \quad \alpha \in (0,\pi), \\
\phi(z,\,\cdot\,,a), \quad \alpha = 0, \end{cases}   \\[2mm]
& \psi_+(z, \, \cdot \,,a,\beta) = c_+ \begin{cases} \big\{
\big[- \sin(\beta) \theta^{[1]}(z,b,a) + \cos(\beta) \theta(z,b,a)\big] \phi(z,\,\cdot\,,a)   \\
+ \big[\sin(\beta) \phi^{[1]}(z,b,a) - \cos(\beta) \phi(z,b,a)\big] \theta(z,\,\cdot\,,a)\big\},  \\
\hspace*{5.95cm} \beta \in (0,\pi), \\
\theta(z,b,a) \phi(z,\,\cdot\,,a) - \phi(z,b,a) \theta(z,\,\cdot\,,a), \quad \beta = 0, \end{cases}
\end{align}
and hence the Green's function of $H_{\alpha,\beta}$ is of the
semi-separable form,
\begin{align}
G_{\alpha,\beta}(z,x,x') & = (H_{\alpha,\beta} - z I_{L_r^2((a,b))})^{-1}(x,x')   \\
&= \f{1}{W(\psi_+(z, \, \cdot \,,a,\beta),\psi_-(z, \, \cdot \,,a,\alpha))}
\no \\
& \quad \times \begin{cases}
\psi_-(z,x,a,\alpha) \psi_+(z,x',a,\beta), & a \leq x \leq x' \leq b, \\
\psi_-(z,x',a,\alpha) \psi_+(z,x,a,\beta), & a \leq  x' \leq x \leq b.
\end{cases}    \no
\end{align}

A direct computation then reveals the following connection between $F_{\alpha, \beta}(z)$, $\psi_-(z, \, \cdot \,,a,\alpha)$, and
$\psi_+(z, \, \cdot \,,a,\beta)$,
\begin{align}
F_{\alpha,\beta}(z) &= \f{1}{W(\psi_+(z, \, \cdot \,,a,\beta),
\psi_-(z, \, \cdot \,,a,\alpha))}       \lb{14.3.42} \\
& \quad \times \begin{cases} \psi_-(z,b,a,0) \psi_+(z,a,a,0),
\quad \alpha = \beta = 0, \\[1mm]
\big[-\sin (\beta ) \psi^{[1]}_-(z,b,a,0) + \cos(\beta) \psi_-(z,b,a,0)\big]
\psi_+(z,a,a,\beta), \\[.5mm]
\hspace*{6.35cm} \alpha = 0, \, \beta \in (0,\pi), \\[1mm]
\psi_-(z,b,a,\alpha)
\big[\sin (\alpha ) \psi^{[1]}_+(z,a,a,0) + \cos(\alpha) \psi_+(z,a,a,0)\big], \\[.5mm]
\hspace*{6.05cm} \alpha \in (0,\pi), \, \beta = 0, \\[1mm]
\big[-\sin (\beta )\psi^{[1]}_-(z,b,a,\alpha) + \cos(\beta) \psi_-(z,b,a,\alpha)\big] \\[.5mm]
\times \big[\sin (\alpha ) \psi^{[1]}_+(z,a,a,\beta) + \cos (\alpha) \psi_+(z,a,a,\beta)\big],
\quad \alpha, \beta \in (0,\pi).
\end{cases}     \no
\end{align}

Combining \eqref{14.3.42} and \eqref{14.3.23} thus yields
\begin{align}
& \tr_{L^2_r((a,b))} \big((H_{\alpha,\beta} - z I_{L^2_r((a,b))})^{-1}\big)   \no \\
& \quad =
\begin{cases}
- (d/dz) \ln\big(\psi_+(z,a,a,0)\big/\psi^{[1]}_+(z,b,a,0)\big) \\[.5mm]
\quad = - (d/dz) \ln\big(\psi_-(z,b,a,0)\big/\psi^{[1]}_-(z,a,a,0)\big),
& \alpha = \beta = 0, \\[1mm]
- (d/dz) \ln(\psi_+(z,a,a,\beta)/\psi_+(z,b,a,\beta)),
& \alpha = 0, \, \beta \in (0,\pi), \\[1mm]
- (d/dz) \ln(\psi_-(z,b,a,\alpha )/\psi_-(z,a,a,\alpha)),
& \alpha \in (0,\pi), \, \beta = 0, \\[1mm]
- (d/dz) \ln\Big(\f{W(\psi_+(z, \, \cdot \,,a,\beta),
\psi_-(z, \, \cdot \,,a,\alpha))}{\psi_+(z,b,a,\beta) \psi_-(z,a,a,\alpha)}\Big),
& \alpha, \beta \in (0,\pi),
\end{cases}     \lb{14.3.43} \\
& \hspace*{8.37cm} z \in \rho(H_{\alpha,\beta}).    \no
\end{align}

\smallskip

Next, applying Theorem \ref{14.t3.4} in the context of Theorem \ref{14.t2.10} immediately yields results
about $\zeta$-regularized determinants (see also \cite{FGK15}, \cite[Chs.~2,3]{Ki02}, \cite{LT98}).

%%%%%%%%
\begin{remark} \lb{14.r3.7}
In Example \ref{14.e3.17} we will consider a simple case with negative eigenvalues present.
Otherwise, in the examples of this section we will always assume that eigenvalues are non-negative.
If that is not the case, an appropriate imaginary part according to Theorem \ref{14.t2.10} has to be included.
\hfill $\diamond$
\end{remark}
%%%%%%%%

To deal with $\zeta$-regularized determinants we now strengthen Hypothesis \ref{14.h3.1} 
and introduce the following assumptions\footnote{The original archive submission, as well as the published version of this article in J. Funct. Anal.~{\bf 276}, 520--562 (2019), had an additional, superfluous assumption $1/r\in L^\infty ((a,b); dx)$ in 
Hypothesis~\ref{14.h3.8}\,$(i)$.~(The same applies to Hypothesis~3.1 in the paper 
\cite{FGKS22}.)~In addition, the condition $1/(pr) \in L^{\infty}((a,b);dx)$ was originally missed in 
Hypothesis~\ref{14.h3.8}\,$(iv)$.} on $p, q, r$:

%%%%%%%%
\begin{hypothesis} \lb{14.h3.8}
Suppose $p, q, r$ satisfy the following conditions: \\
$(i)$ \,\, $r>0$ a.e.\ on $(a,b)$, $r\in L^1((a,b); dx)$. \\
$(ii)$ \, $p>0$  a.e.\ on $(a,b)$, $1/p\in L^1((a,b); dx)$. \\
$(iii)$ $q\in L^1((a,b); dx)$, $q$ is real-valued a.e.\ on $(a,b)$.\\
$(iv)$ $p \,r$ and $(p\,r)'/r$ are absolutely continuous on $[a,b]$, and $1/(pr) \in L^{\infty}((a,b);dx)$.
\end{hypothesis}
%%%%%%%%

The substitutions (cf.\ \cite[p.~2]{LS75})
\begin{align}
&\xi(x)=\dfrac{1}{c}\int_a^x dt\ [r(t)/p(t)]^{1/2},\quad \xi(x) \in [0,1] \, \text{ for } \, x \in [a,b],   \\
& \xi'(x) = c^{-1} [r(x)/p(x)]^{1/2} > 0 \, \text{ a.e.~on $(a,b)$,}     \\
&u(\xi)=[p(x(\xi))r(x(\xi))]^{1/4}y(x(\xi)),     
\end{align}
where $c$ is given by
\begin{equation}
c= \int_a^b dt \, [r(t)/p(t)]^{1/2},
\end{equation}
transform the Sturm--Liouville problem $(\tau y)(x)=zy(x)$ into
\begin{align}\lb{3.17}
- u''(\xi)+V(\xi)u(\xi)=c^2 zu(\xi),
\end{align}
and abbreviating 
\begin{equation} 
\mu(\xi)=[p(x(\xi))r(x(\xi))]^{1/4},
\end{equation}
one verifies that
\begin{align}
\begin{split}
V(\xi)&=\dfrac{\mu''(\xi)}{\mu(\xi)}+c^2\dfrac{q(x)}{r(x)}\\
&=-\dfrac{c^2}{16}\dfrac{1}{p(x)r(x)}\left[\dfrac{(p(x)r(x))'}{r(x)}\right]^2+\dfrac{c^2}{4}\dfrac{1}{r(x)}\dfrac{d}{dx}\left[\dfrac{(p(x)r(x))'}{r(x)}\right]+c^2\dfrac{q(x)}{r(x)}.
\end{split}
\end{align}
Hypothesis \ref{14.h3.8} guarantees\footnote{We are indebted to C.\ Bennewitz \cite{Be17} for a very helpful discussion of this issue.} that 
\begin{equation} 
V\in L^1((0,1);d\xi), 
\end{equation} 
and as a consequence one has asymptotically for $x \in (a,b]$,
\begin{align}
\begin{split}
\phi (z,x,a) \underset{|z|\to\infty}{=}& 2^{-1} z^{-1/2} [p(a) r(a) p(x) r(x)]^{-1/4}    \\
& \times \exp\bigg(z^{1/2} \int_a^x dt \, [r(t)/p(t)]^{1/2}\bigg)
 \big( 1+ \Oh\big(|z|^{-1/2}\big)\big),       \\
 \phi^{[1]} (z,x,a) \underset{|z|\to\infty}{=}& 2^{-1} [p(a) r(a)]^{-1/4} [p(x) r(x)]^{1/4}    \\
& \times \exp\bigg(z^{1/2} \int_a^x dt \, [r(t)/p(t)]^{1/2}\bigg)
 \big( 1+ \Oh\big(|z|^{-1/2}\big)\big),       \\
\theta (z,x,a) \underset{|z|\to\infty}{=}& 2 ^{-1} [r(a)/p(a)]^{1/2} [p(a) r(a) p(x) r(x)]^{-1/4}    \\
& \times
\exp\bigg(z^{1/2} \int_a^x dt \, [r(t)/p(t)]^{1/2}\bigg)
\big( 1+ \Oh\big(|z|^{-1/2}\big)\big),     \\
\theta^{[1]} (z,x,a) \underset{|z|\to\infty}{=}& 2 ^{-1} z^{1/2} p(a)^{-1} [p(a) r(a) p(x) r(x)]^{1/4}     \\
& \times
\exp\bigg(z^{1/2} \int_a^x dt \, [r(t)/p(t)]^{1/2}\bigg)
\big( 1+ \Oh\big(|z|^{-1/2}\big)\big).
\end{split}
\end{align}
This asymptotic behavior is used to guarantee that assumption \eqref{14.2.31} is satisfied in several of the following examples. 

%%%%%%%%
\begin{theorem} \lb{14.t3.8}
Assume Hypothesis \ref{14.h3.8} and denote by $H_{\alpha , \beta, j}$ and $H_{\varphi , R, j}$, $j=1,2$,
the self-adjoint extensions of
$H_{min}$ as described in Theorem \ref{14.t3.3}\,$(i)$ and $(ii)$, respectively. Here the index $j$ refers to a potential $q_j$ in \eqref{14.3.1}, $j=1,2$.
Then the following items $(i)$ and $(ii)$ hold:\\
$(i)$
\begin{equation}
\zeta ' \left(0; H_{\alpha , \beta , 1}, H_{\alpha , \beta , 2} \right) = - \lim_{t\downarrow 0} \ln \bigg(\left| \frac{ F_{\alpha , \beta , 2} (t e^{i\theta})}{F_{\alpha , \beta , 1} (t e^{i\theta})} \,\,
 (te^{i\theta})^{[m(0, H_{\alpha , \beta , 1}) -m (0, H_{\alpha , \beta , 2})]}\right|\bigg).
 \end{equation}
$(ii)$
\begin{equation} \zeta ' \left(0; H_{\varphi , R , 1}, H_{\varphi , R , 2} \right)
= - \lim_{t\downarrow 0} \ln \bigg(\left| \frac{ F_{\varphi ,R , 2} (t e^{i\theta})}{F_{\varphi , R , 1} (t e^{i\theta})} \,\,
 (te^{i\theta})^{[m(0, H_{\varphi , R , 1}) -m (0, H_{\varphi , R , 2})]}\right|\bigg).
 \end{equation}
\end{theorem}
%%%%%%%%
\begin{proof}
This follows immediately from Theorem \ref{14.t2.8} and Theorem \ref{14.t2.10} applied to $S_j = H_{\alpha , \beta , j}$, respectively, $S_j = H_{\varphi , R, j}$, $j=1,2$.
\end{proof}
%%%%%%%%

%%%%%%%%
\begin{remark} \lb{14.r3.9}
In the absence of zero eigenvalues of $H_{\alpha , \beta , j}$, respectively, $H_{\varphi , R, j}$, $j=1,2$, these results simplify to
\begin{equation}
\zeta ' \left( 0; H_{\alpha , \beta , 1}, H_{\alpha , \beta , 2} \right) = \ln \bigg(\left|\frac{ F_{\alpha , \beta , 1} (0) } { F_{\alpha , \beta , 2} (0)  } \right|\bigg),
\end{equation}
respectively,
\begin{equation}
\zeta ' \left( 0; H_{\varphi , R , 1}, H_{\varphi , R , 2} \right) = \ln \bigg(\left|\frac{ F_{\varphi, R , 1} (0) } { F_{\varphi , R , 2} (0) } \right|\bigg).
\end{equation}
In case there are zero eigenvalues, a suitable energy shift will again lead to this case.
\hfill $\diamond$
\end{remark}
%%%%%%%%

%%%%%%%%
\begin{example} \lb{14.e3.10}
If one of the potentials vanishes, say $q_1=0$, more explicit results can be obtained. We consider separated boundary conditions with no zero eigenvalues present for $j=1, 2$.
Then
\begin{align}
\begin{split}
F_{\alpha , \beta , j} (0) & = \cos (\alpha) \left[ - \sin (\beta) \phi_j^{[1]} (0,b,a) + \cos (\beta) \phi_j (0,b,a) \right]  \\
 & \quad - \sin (\alpha) \left[ - \sin (\beta) \theta_j^{[1]} (0,b,a) + \cos (\beta) \theta _j (0,b,a) \right],
\quad  j=1,2.
 \end{split}
 \end{align}
For $q_1 =0$, $\theta _1 (0,x,a)$ and $\phi_1 (0,x,a)$ satisfy the initial value problems
\begin{align}
\begin{split}
- \bigg(\frac d {dx} p(x) \frac d {dx} \bigg) \theta_1 (0,x,a) &=0, \quad \quad \theta_1 (0,a,a) =1, \quad \theta _1 ^{[1]} (0,a,a) =0, \\\
 - \bigg(\frac d {dx} p(x) \frac d {dx} \bigg) \phi_1 (0,x,a) &=0, \quad \quad \phi_1 (0,a,a) =0, \quad \phi _1 ^{[1]} (0,a,a) =1,
\end{split}
\end{align}
solutions of which are
\begin{equation}
\phi_1 (0,x,a) = \int_a^x dt \, p(t)^{-1}, \quad \theta_1 (0,x,a) =1. \lb{14.3.p=r=1}
\end{equation}
Then one can show that
\begin{equation}
F_{\alpha , \beta , 1} (0) = - \sin (\alpha + \beta ) + \cos (\alpha) \cos (\beta) \int_a^b dt \, p(t)^{-1}.
\end{equation}
The relative $\zeta$-regularized determinant in this case then reads
\begin{equation}
\zeta ' \left( 0; H_{\alpha , \beta , 1} , H_{\alpha , \beta , 2} \right)
= \ln \Bigg(\left| \frac{ \cos (\alpha) \cos (\beta) \int_a^b dt \, p(t)^{-1} - \sin (\alpha + \beta) }
{F_{\alpha , \beta , 2} (0) } \right|\Bigg).\lb{14.3.prgen}
\end{equation}
\end{example}
%%%%%%%

%%%%%%%
\begin{example} \lb{14.e3.11}
Restricting Example \ref{14.e3.10} to the case $p(x) = r(x) =1$, the $\zeta$-determinant
$\zeta ' (0; H_{\alpha , \beta , 2})$ can be computed explicitly. First, one notes that
under Hypothesis \ref{14.h3.8}, the $\zeta$-function can be analytically continued to a neighborhood
of $s=0$ and $\zeta ' (0; H_{\alpha , \beta, 2} ) $ is
a well-defined quantity; similarly, $\zeta ' (0; H_{\alpha , \beta , 1})$ is well-defined for $q_1 =0$.
Computing $\zeta ' (0; H_{\alpha , \beta , 1})$, employing \eqref{14.3.prgen}, one obtains
$\zeta ' (0; H_{\alpha , \beta , 2})$.

First we assume there are no zero eigenvalues for $j=1$, which is the case if
\begin{equation}
(b-a) \cos (\alpha) \cos (\beta) - \sin (\alpha + \beta ) \neq 0.
\end{equation}
Next, one notes that
\begin{equation}
\phi_1 (z,x,a) = z^{-1/2} \sin \big( z^{1/2} (x-a) \big), \quad
\theta _1 (z,x,a) = \cos \big( z^{1/2} (x-a)\big),
\end{equation}
and thus from \eqref{14.3.21},
\begin{align}
F_{\alpha , \beta , 1} (z) &= \cos (\alpha) \big[ - \sin (\beta) \cos \big( z^{1/2} (b-a) \big)
+ z^{-1/2} \cos (\beta)  \sin \big( z^{1/2} (b-a) \big) \big]     \no \\
& \quad - \sin (\alpha) \big[z^{1/2} \sin (\beta) \sin \big( z^{1/2} (b-a) \big)
+ \cos (\beta) \cos \big( z^{1/2} (b-a) \big) \big].    \lb{14.3.fab1z}
\end{align}
The details of what follows depend on the boundary conditions imposed. In this example, we consider $\alpha , \beta \in (0,\pi )$.
Along the relevant contour, as $|z|\to\infty$, one has $\Im \big(z^{1/2}\big) >0$ and the asymptotics for the boundary conditions considered reads
\begin{equation}
F_{\alpha , \beta , 1} (z) = F_{\alpha , \beta , 1} ^{asym} (z) \big[ 1+ \Oh\big( z^{-1/2} \big) \big],
\end{equation}
where
\begin{equation}
F_{\alpha , \beta , 1} ^{asym} (z) = - (i /2) z^{1/2} \sin(\alpha) \sin(\beta) e^{-i z^{1/2} (b-a)} .
\lb{14.3.11asym}
\end{equation}
Adding and subtracting this asymptotics where applicable, we rewrite the $\zeta$-function
for $H_{\alpha , \beta , 1}$ in the form
\begin{align}
\zeta (s; H_{\alpha , \beta , 1} ) &= e^{is (\pi - \theta )} \frac{\sin (\pi s)} \pi \int_0^1 dt \, t^{-s}
\frac d {dt} \ln \big(F_{\alpha , \beta , 1} \left(t e^{i\theta} \right)\big)  \no \\
& \quad +  e^{is (\pi - \theta )} \frac{\sin (\pi s)} \pi \int_1^\infty dt \, t^{-s} \frac d {dt}
\ln \bigg(\frac{F_{\alpha , \beta , 1} \left(t e^{i\theta} \right)} {F_{\alpha , \beta , 1}^{asym} \left(te^{i\theta } \right)}\bigg)    \no \\
& \quad + \zeta ^{asym} (s; H_{\alpha , \beta , 1} ),\lb{14.3.subasym}
\end{align}
where
\begin{equation}
\zeta ^{asym} (s; H_{\alpha , \beta , 1} ) =  e^{is (\pi - \theta )} \frac{\sin (\pi s)} \pi \int_1^\infty dt \,
t^{-s} \frac d {dt} \ln \big(F_{\alpha , \beta , 1}^{asym} \left(te^{i\theta } \right)\big).  \lb{14.3.70}
\end{equation}
This representation is valid for $\frac 1 2 < \Re (s) < 1$. The term \eqref{14.3.70} is easily computed and yields
\begin{equation}
\zeta^{asym} (s;H_{\alpha , \beta , 1} ) = e^{is(\pi - \theta )} \frac{\sin (\pi s)} \pi \left[ \frac 1 {2s} - \frac{i} 2 e^{i \theta/ 2} \f{(b-a)}{s- (1/2)} \right],
\end{equation}
yielding its analytic continuation to $-\frac 1 2 < \Re (s) < 1$.
The $\zeta$-determinant for $H_{\alpha , \beta , 1}$ then follows from
\begin{align}
\zeta ' (0; H_{\alpha , \beta , 1} ) &= \Re \left( \ln \big(F_{\alpha , \beta , 1} ^{asym} \left(e^{i\theta}\big) \right) -
\ln (F_{\alpha , \beta , 1} (0)) + \zeta ^{asym \,\, \prime}  (0; H_{\alpha , \beta , 1} ) \right)     \no \\
&=  - \ln \bigg(\left| \frac{ 2 \left[ \cos (\alpha) \, \cos (\beta) \,\,(b-a) - \sin (\alpha + \beta ) \right] }{\sin(\alpha)  \sin(\beta) } \right|\bigg).     \lb{14.3.albe0fi}
\end{align}
Using expression \eqref{14.3.albe0fi} in \eqref{14.3.prgen} yields
\begin{equation}
\zeta ' (0; H_{\alpha , \beta , 2} ) = \ln \bigg(\left|\frac{ \sin(\alpha)  \sin(\beta) }
{2 F_{\alpha , \beta , 2} (0) } \right|\bigg).\lb{14.3.sipofi}
\end{equation}
If there is a zero eigenvalue for $j=1$, namely, if
\begin{equation}
(b-a) \cos (\alpha) \cos (\beta)  - \sin (\alpha + \beta ) =0,
\end{equation}
the relevant formula for the relative $\zeta$-determinant is
\begin{equation}
\zeta ' (0; H_{\alpha , \beta , 1} , H_{\alpha , \beta , 2} ) = \ln \bigg(\left| \frac{ F_{\alpha , \beta , 1} ' (0) } {F_{\alpha , \beta , 2} (0) }\right|\bigg).
\end{equation}
Employing \eqref{14.3.fab1z}, this can be cast in the form
\begin{equation}
\zeta ' (0; H_{\alpha , \beta , 1}, H_{\alpha , \beta , 2} ) = \ln \bigg(\left| \frac{ (b-a) \left[ (b-a) \sin (\alpha + \beta ) -3 \sin(\alpha)  \sin \beta\right]}{3 F_{\alpha , \beta , 2} (0) } \right|\bigg).
\end{equation}
In order to compute $\zeta ' (0; H_{\alpha , \beta , 2} )$, the part $\zeta ' (0; H_{\alpha , \beta , 1} )$ can be computed as before, the only difference
being that $F_{\alpha , \beta , 1} (z)$ is replaced by $F_{\alpha , \beta , 1} (z) /z$. The final answer then reads
\begin{equation}
\zeta ' (0; H_{\alpha , \beta , 1}) = - \ln \bigg(\left| 2 (b-a) \left[ 1-\frac{(b-a) \sin (\alpha + \beta )}{3 \sin(\alpha)  \sin \beta} \right] \right|\bigg).
\lb{14.3.78}
\end{equation}
$($This confirms the known result in the case of Neumann boundary conditions at $x=a$ and $x=b$, that is, for
$\alpha = \beta = \pi/2$.$)$ From \eqref{14.3.78} it is immediate that
\begin{equation}
\zeta ' (0; H_{\alpha , \beta , 2} ) = \ln \bigg(\left| \frac{  \sin(\alpha)  \sin(\beta) } {2 F_{\alpha , \beta , 2} (0) }\right|\bigg). \lb{14.3.ab22}
\end{equation}
\end{example}
%%%%%%%

%%%%%%%
\begin{example} \lb{14.e3.12}
Next, consider the Dirichlet boundary condition at $x=a$ and a Robin boundary condition at $x=b$, that is, $\alpha =0$ and $\beta \in (0, \pi )$. In the absence of zero eigenvalues one then has
\begin{equation}
\zeta ' (0; H_{0, \beta , 1}, H_{0, \beta , 2} ) = \ln \bigg(\left| \frac{ F_{0, \beta , 1} (0) } { F_{0, \beta , 2}(0) } \right|\bigg).
\end{equation}
In order to find the $\zeta$-determinant $\zeta ' (0; H_{0, \beta , 2} )$ we consider as before $q_1 =0$. Then,
\begin{equation}
F_{0, \beta , 1} (z) = - \sin(\beta) \cos \big( z^{1/2} (b-a) \big)
+ z^{-1/2} \cos \big(\beta) \sin (z^{1/2} (b-a) \big) .
\end{equation}
The relevant asymptotic large-$|z|$ behavior is
\begin{equation}
F_{0, \beta , 1} ^{asym} (z) = - \sin(\beta) e^{-i z^{1/2} (b-a)}/2,
\end{equation}
and proceeding along the lines of previous computations, one finds
\begin{equation}
\zeta ' (0; H_{0, \beta , 1} )
= \ln \bigg(\left| \frac{ \sin(\beta) } {2 \left[ \sin(\beta) -(b-a) \cos (\beta) \right] } \right|\bigg).  \lb{14.3.83}
\end{equation}
$($Again, this confirms the known case where $\beta = \pi /2$$)$. From \eqref{14.3.83} it is immediate that
\begin{equation}
\zeta ' (0; H_{0, \beta , 2} ) = \ln \bigg(\left| \frac{  \sin(\beta) }{2 F_{0, \beta , 2} (0) }\right|\bigg).
\lb{14.3.ab23}
\end{equation}
\end{example}
%%%%%%%%

%%%%%%%%
\begin{example} \lb{14.e3.13}
For Dirichlet boundary conditions at both endpoints, that is, for $\alpha = \beta =0$, the relative
zeta-determinant follows from
\begin{equation}
\zeta ' (0; H_{0,0,1} , H_{0,0,2} ) = \ln \bigg(\left| \frac{ F_{0,0,1} (0)}{F_{0,0,2} (0)} \right|\bigg).
\end{equation}
For $q_1 =0$, from
\begin{equation}
F_{0,0,1} (z) = z^{-1/2} \sin \big(z^{1/2} (b-a)\big),
\end{equation}
the relevant asymptotics is
\begin{equation}
F_{0,0,1}^{asym} (z) = (i/2) z^{-1/2} e^{-i z^{1/2} (b-a)}.
\end{equation}
Thus, one finds
\begin{equation}
\zeta ' (0;H_{0,0,1}) = - \ln (2 (b-a)),
\end{equation}
and hence,
\begin{equation}
\zeta ' (0; H_{0,0,2}) = - \ln (|2F_{0,0,2} (0)|). \lb{14.3.ab24}
\end{equation}
\end{example}
%%%%%%%%

%%%%%%%%
\begin{remark} \lb{14.r3.14}
Under the assumptions of Example \ref{14.e3.11} no additional computations are needed when
considering certain relative determinants for different boundary conditions. Indeed,
for $\alpha_j,\beta_j\in (0,\pi)$, $j=1,2$, eq.~\eqref{14.3.sipofi} is valid replacing $\alpha,\beta\to \alpha_1,\beta_1$ and $F_{\alpha ,\beta,2} \to F_{\alpha_2,\beta_2,2}$, and
eq.~\eqref{14.3.ab23} is valid replacing $\beta \to \beta_1$, and $F_{0,\beta,2} \to F_{0,\beta_2,2}$.
\hfill $\diamond$
\end{remark}
%%%%%%%%

%%%%%%%%
\begin{example} \lb{14.e3.15}
As an example for coupled boundary conditions we reconsider the Krein--von Neumann extension, Example \ref{14.e3.6}.
We first note that different potentials $q_1 \neq q_2$ lead to different Krein--von Neumann extensions $R_{K_1} \neq R_{K_2}$; see, for instance, \eqref{14.3.31}. Nevertheless,
under the assumptions made, Theorem \ref{14.t3.8}\,$(ii)$ remains valid and
\begin{equation}
\zeta ' (0; H_{0,R_{K_1},1},H_{0,R_{K_2},2}) = \lim_{z\to 0} \ln \bigg(\left| \frac{F_{0,R_{K_1},1} (z) }{F_{0,R_{K_2},2} (z)} \right|\bigg),
\end{equation}
and from \eqref{14.3.37} one finds
\begin{equation}
\zeta ' (0; H_{0,R_{K_1},1},H_{0,R_{K_2},2}) =  \ln (| c_1/c_2|) ,\lb{14.3.kvndet1}
\end{equation}
with
\begin{equation}
c_j = \frac 1 2 \left( \dot \theta_j^{[1]} (0,b,a) \,\, \dot \phi _j (0,b,a) - \dot \theta_j (0,b,a) \dot \phi_j^{[1]} (0,b,a)\right), \quad j = 1,2.      \lb{14.3.kvndet2}
\end{equation}

For the case of vanishing potential, $q_1=0$, the constant $c_1$ can be determined explicitly. To this end we need the small-$z$ expansion of the solutions of
\begin{equation}
- \bigg(\frac 1 {r(x)} \frac d {dx} p(x) \frac d {dx} \bigg) \theta_1 (z,x,a) = z \theta_1 (z,x,a),
\quad \theta_1 (z,a,a) =1, \,\, \theta_1 ^{[1]} (z,a,a) =0, \lb{14.3.kvnszt}
\end{equation}
and
\begin{equation}
- \bigg(\frac 1 {r(x)} \frac d {dx} p(x) \frac d {dx} \bigg) \phi_1 (z,x,a) = z \phi_1 (z,x,a), \quad \phi_1 (z,a,a) =0, \,\, \phi_1 ^{[1]} (z,a,a) =1. \lb{14.3.kvnszp}
\end{equation}
Expanding
\begin{equation}
\theta _1 (z,x,a) = \theta _1 (0,x,a) + z \dot \theta _{1} (0,x,a) + \Oh\big(z^2\big),
\end{equation}
one compares $ \Oh(z)$-terms in \eqref{14.3.kvnszt} to find with \eqref{14.3.p=r=1}
\begin{equation}
\bigg(\frac d {dx} p(x) \frac d {dx} \bigg) \dot \theta_1 (0,x,a) = - r(x) .   \lb{14.3.kvndet3}
\end{equation}
Integrating, this yields
\begin{equation}
\dot \theta_{1} ^{[1]} (0,x,a) - \dot \theta _{1} ^{[1]} (0,a,a) = - \int_a^x du \,r(u),
\end{equation}
but given $\theta_1 ^{[1]} (z,a,a) =0$, one concludes that
\begin{equation}
\dot \theta_1 ^{[1]} (0,x,a) = - \int_a^x du \, r(u).    \lb{14.3.kvndet4}
\end{equation}
Similarly, integrating \eqref{14.3.kvndet4},
\begin{equation}
\dot \theta_{1} (0,b,a) = - \int_a^b dv \, p(v)^{-1} \int_a^v du \, r(u).
\end{equation}
Proceeding in the same way for $\phi_1 (z,x,a)$, one first shows
\begin{equation}
\bigg(\frac d {dx} p(x) \frac d {dx} \bigg) \dot \phi_1 (0,x,a) = - r(x) \int_a^x du \, p(u)^{-1},
\end{equation}
and thus
\begin{equation}
\dot \phi_{1} ^{[1]} (0,x,a) = - \int_a^x dv \, r(v) \int_a^v du \, p(u)^{-1}.
\end{equation}
Furthermore,
\begin{equation}
\dot \phi_{1} (0,x,a) = - \int_a^x dw \, p(w)^{-1} \int_a^w dv \, r(v) \int_a^v du \, p(u)^{-1}.
\end{equation}
Altogether this yields
\begin{align}
\begin{split}
c_1 &= \frac 1 2 \left[ \left( \int_a^b dt \, r(t) \right) \left( \int_a^b dw \int_a^wdv \int_a^v du \,\, \frac{ r(v)}{p(w) p(u)} \right) \right.    \\
 & \quad \left. - \left( \int_a^b dv \int_a^v du \,\, \frac{r(u)}{p(v)} \right) \left( \int_a^bdw \int_a^wdt \, \frac{ r(w)} {p(t)} \right)\right].       \lb{14.3.cjclosed}
 \end{split}
\end{align}
\end{example}
%%%%%%%%

%%%%%%%%
\begin{example} \lb{14.e3.16}
For the particular case $r(x) = p(x) =1$, we now recompute the $\zeta$-determinant
$\zeta ' (0; H_{0,R_{K_2},2})$ by choosing $q_1=0$. First, from \eqref{14.3.cjclosed} one finds that 
\begin{equation}
c_1 = - (b-a)^4/24.
\end{equation}
From Example \ref{14.e3.6} one determines
\begin{equation}
F_{0,R_{K_1},1} (z) = 2 \big[1-\cos \big(z^{1/2} (b-a)\big)\big] - (b-a) z^{1/2} \sin \big( z^{1/2} (b-a)\big).
\end{equation}
The relevant leading asymptotics reads
\begin{equation}
F_{0,R_{K_1},1}^{asym} (z) = - i (b-a) z^{1/2} e^{-i z^{1/2} (b-a)}\big/2.
\end{equation}
The $zeta$-function for the Krein--von Neumann extension is therefore analyzed using
\begin{align}
\zeta (s;H_{0,R_{K_1},1}) &= e^{is (\pi - \theta) } \frac{\sin (\pi s)} \pi \int_0^\infty dt \, t^{-s} \frac d {dt} \ln\bigg(\frac{ F_{0,R_{K_1},1} (te^{i\theta})}{t^2 e^{2i\theta}}\bigg)       \no \\
&= e^{is (\pi - \theta) } \frac{\sin (\pi s)} \pi \int_0^1 dt \, t^{-s} \frac d {dt} \ln \left(\frac{ F_{0,R_{K_1},1} (te^{i\theta})}{t^2 e^{2i\theta}}\right)      \no \\
& \quad + e^{is (\pi - \theta) } \frac{\sin (\pi s)} \pi \int_1^\infty dt \, t^{-s} \frac d {dt}
\ln \left(\frac{ F_{0,R_{K_1},1} (te^{i\theta})}{F_{0,R_{K_1},1}^{asym}(t e^{i\theta})}\right)    \no \\
& \quad +\zeta^{asym}(s;H_{0,R_{K_1},1}), \label{kreinrpe}
\end{align}
where
\begin{align}
\zeta^{asym}(s;H_{0,R_{K_1},1}) &=
 e^{is (\pi - \theta) } \frac{\sin (\pi s)} \pi \int_1^\infty dt \, t^{-s} \frac d {dt}
 \ln \left(\frac{ F_{0,R_{K_1},1} ^{asym}(t e^{i\theta})}{t^2 e^{2i\theta}}\right)      \no \\
&= e^{is (\pi - \theta) } \frac{\sin (\pi s)} \pi \left( - \frac 3 {2s}
-\frac i 2 (b-a) e^{i \theta/2 } \frac{1}{s- (1/2)} \right),
\end{align}
such that
\begin{equation}
{\zeta ^{asym}} ' (0;H_{0,R_{K_1},1}) = i (b-a) e^{i\theta /2} + [3 i (\theta - \pi )/2].
\end{equation}
From (\ref{kreinrpe}) we then find
\begin{equation}
\zeta' (0; H_{0,R_{K_1},1} ) = - \ln \big((b-a)^3/6\big),
\end{equation}
in agreement with \cite{MCKB15}. Finally, this proves
\begin{equation}
\zeta ' (0; H_{0,R_{K_2},2}) = \zeta ' (0;H_{0,R_{K_1},1}) + \ln (|c_1/c_2|)
= \ln \bigg(\left| \frac{ b-a} {4c_2}\right|\bigg) .\lb{14.3.kvndetfinal}
\end{equation}
\end{example}
%%%%%%%%

%%%%%%%%
\begin{example} \lb{14.e3.17}
As our final example we consider a case where negative eigenvalues are present. Let $x\in (0, \pi )$
and $\tau_j = - (d^2/dx^2) - m_j^2$ with $m_j \in (n_j , n_j +1 )$, $n_j \in \bbN$, $j=1,2$.
Imposing Dirichlet boundary conditions at both endpoints, the eigenvalues are
$\lambda_\ell ^{(j)} = \ell^2 - m_j^2$, $\ell \in \bbN$, such that there are $n_j$ negative eigenvalues for $H_{0,0,j}$. The $\zeta$-function representation for each $j= 1,2$, can be found following
Examples \ref{14.e3.11} and \ref{14.e3.13}. We note that
\begin{align}
F_{0,0,j} (z) &= \big(z+m_j^2\big)^{-1/2} \sin \Big(\big(z+m_j^2\big)^{1/2} \pi \Big)   \no \\
&= \big(z+m_j^2\big)^{-1/2} (i/2)  e^{-i (z+m_j^2)^{1/2} \pi} \left( 1-e^{2 \pi i (z+m_j^2)^{1/2}}\right)
\no \\
&=: F_{0,0,j}^{asym} (z) \left( 1-e^{2 \pi i (z+m_j^2)^{1/2}}\right),
\end{align}
and hence,
\begin{equation}
\zeta (s; H_{0,0,j}) = e^{is (\pi -\theta )} \frac{\sin (\pi s)} \pi \int_0^\infty dt \, t^{-s} \frac d {dt} \ln \bigg(\frac{ F_{0,0,j} \left( te^{i\theta}\right)}{F_{0,0,j}^{asym} \left( te^{i\theta}\right)}\bigg)
+\zeta ^{asym} (s; H_{0,0,j}) ,
\end{equation}
where
\begin{equation}
\zeta^{asym} (s; H_{0,0,j} ) = e^{is (\pi -\theta )} \frac{\sin (\pi s)} \pi \int_0^\infty dt \, t^{-s} \frac d {dt} \ln \big(F_{0,0,j}^{asym} \big( t e^{i\theta} \big)\big),
\end{equation}
this representation being valid for $ (1/2) < \Re (s) < 1$.
From \cite[3.193, 3.194]{GR80} one infers
\begin{align}
& \int_0^\infty dt \,\, \frac{t^{-s}}{(t+M)^{1/2}} = \frac{M^{(1/2) -s} \Gamma (1-s) \Gamma \left( s-\frac 1 2 \right)} {\pi^{1/2}} ,  \\
& \int_0^\infty dt \,\, \frac{t^{-s}}{t+M} = \frac \pi {M^s \sin (\pi s)} , \label{14.3.intform}
\end{align}
and hence one obtains
\begin{equation}
\zeta^{asym} (s; H_{0,0,j}) = - \frac{ e^{is\pi}}{2m_j^{2s}} \left[ \frac{ i \pi^{1/2} m_j
\Gamma \left( s- (1/2) \right)}{\Gamma (s)} +1\right],
\end{equation}
implying
\begin{equation}
\zeta^{asym \prime} (0; H_{0,0,j}) = - i (\pi/2) + im_j\pi + \ln (m_j) .
\end{equation}
It then follows that
\begin{equation}
\zeta ' (0; H_{0,0,j}) = - \ln \left( 1-e^{2\pi im_j} \right) - i (\pi/2) + im_j\pi + \ln (m_j). \lb{14.3.prefin}
\end{equation}
In order to obtain the final answer explicitly, showing the relation between the imaginary part and the number of negative eigenvalues, we first note that
\begin{equation}
1-e^{2\pi im_j} = - 2 i e^{\pi i m_j} \sin (\pi m_j).
\end{equation}
A careful analysis of the argument of $1-e^{2\pi im_j}$ then shows that it equals 
$\pi [m_j-n_j- (1/2)]$, such that
\begin{equation}
\zeta ' (0;H_{0,0,j}) = i\pi n_j - \ln \bigg(\bigg| \frac{ 2 \sin (\pi m_j)}{m_j}\bigg|\bigg).   \lb{14.3.negevdet1}
\end{equation}
Considering instead
\begin{equation}
\zeta (s; H_{0,0,1},H_{0,0,2}) = e^{is (\pi - \theta )} \frac{ \sin (\pi s)} \pi \int_0^\infty dt \, t^{-s} \frac d {dt}
\ln \bigg(\frac{ F_{0,0,2} (t e^{i\theta})}{F_{0,0,1} (t e^{i\theta } )}\bigg),  \lb{14.3.negevdet2}
\end{equation}
in \eqref{14.3.negevdet1}, one obtains
\begin{equation}
\zeta ' (0; H_{0,0,1}, H_{0,0,2}) = i \pi (n_2-n_1) + \ln \bigg(\left| \frac{ \sin (\pi m_1)}{\sin (\pi m_2)} \,\, \frac{ m_2}{m_1} \right|\bigg).    \lb{14.3.negevdet3}
\end{equation}
The real part of this answer is readily reproduced from \eqref{14.2.33} in Theorem \ref{14.t2.10}.
However, even for this simple example, the behavior of $F_{0,0,1}(t e^{i\theta})/F_{0,0,2} (t e^{i\theta})$
along the integration range $t\in [0,\infty )$ is quite intricate so that finding the correct imaginary part from \eqref{14.3.negevdet2}, namely, from
\begin{equation}
\zeta ' (0; H_{0,0,1},H_{0,0,2}) = \int_0^\infty dt \, \frac d {dt} \ln \bigg(\frac{ F_{0,0,1} (t e^{i\theta})}{F_{0,0,2} (t e^{i\theta } )}\bigg),
\end{equation}
is rather involved, and appears to be next to impossible for more general cases.
\end{example}
%%%%%%%

For the case of Schr\"odinger operators with strongly singular potentials at one or both endpoints of a bounded interval, see \cite{KLP06}, \cite{Le98}, \cite{LV11}, \cite{Sp05}.

%%%%%%%%%%%%%%%%%%%%%%%%%%
%%%%%%%%%%%%%%%%%%%%%%%%%%
\section{Schr\"odinger Operators on a Half-Line} \lb{14.s4}
%%%%%%%%%%%%%%%%%%%%%%%%%%
%%%%%%%%%%%%%%%%%%%%%%%%%%

In our final section we illustrate some of the abstract notions in
Section \ref{14.s2} with the help of self-adjoint Schr\"odinger operators on the half-line
$\bbR_+ = (0,\infty)$. We will focus on the case of short-range potentials $q$ (cf.\ \eqref{14.4.1})
and hence the scattering theory situation which necessitates a comparison with the case $q = 0$
and thus illustrates the case of relative perturbation determinants, relative $\zeta$-functions, and
relative $\zeta$-function regularized determinants.

We assume that the potential coefficient $q$ satisfies the following conditions.

%%%%%%%
\begin{hypothesis} \lb{14.h4.1}
Suppose $q$ satisfies the short-range assumption
\begin{equation}
q\in L^1(\bbR_+; (1 + |x|) dx), \, \text{ $q$ is real-valued a.e.\ on $\bbR_+$.}      \lb{14.4.1}
\end{equation}
\end{hypothesis}
%%%%%%%

Given Hypothesis \ref{14.h4.1}, we take $\tau_+$ to be the Schr\"odinger  differential expression
\begin{equation} \lb{14.4.2}
\tau_+= -\frac{d^2}{dx^2} + q(x) \, \text{ for a.e.~$x\in \bbR_+$,}
\end{equation}
and note that $\tau_+$ is regular at $0$ and in the limit point case at $+\infty$.
The \emph{maximal operator}
$H_{+,max}$ in $L^2(\bbR_+;dx)$ associated with $\tau_+$ is defined by
\begin{align}
&H_{+,max} f=\tau_+ f,    \no \\
& \, f \in \dom(H_{+,max})= \big\{g\in L^2(\bbR_+;dx) \, \big| \, g,  g' \in AC([0,b]) \, \text{for all $b > 0$};     \lb{14.4.3} \\
& \hspace*{7.85cm} \tau_+ g\in  L^2(\bbR_+;dx)\big\},   \no
\intertext{while the \emph{minimal operator} $H_{+,min}$ in
$L^2(\bbR_+;dx)$ associated with
$\tau_+$ is given by}
&H_{+,min} f=\tau_+ f,    \no \\
& \, f \in \dom(H_{+,min})= \big\{g\in L^2(\bbR_+;dx) \, \big| \, g,  g' \in
AC([0,b]) \, \text{for all $b > 0$};   \lb{14.4.4} \\
&\hspace*{5.15cm}g(0)=g'(0)=0; \, \tau_+ g\in L^2(\bbR_+;dx)\big\}.      \no
\end{align}

Again, one notes that the operator $H_{+,min}$ is symmetric and that
\begin{equation}
H_{+,min}^*=H_{+,max}, \quad H_{+,max}^*=H_{+,min}.     \lb{14.4.5}
\end{equation}
Moreover, all self-adjoint extensions of $H_{+,min}$ are given by the
one-parameter family $H_{+,\alpha}$ in $L^2(\bbR_+;dx)$,
\begin{align}
&H_{+,\alpha} f=\tau_+ f,    \no \\
& \, f \in \dom(H_{+,\alpha})= \big\{g\in L^2(\bbR_+;dx) \, \big| \, g,  g' \in
AC([0,b]) \, \text{for all $b > 0$};   \lb{14.4.6} \\
&\hspace*{3.05cm}\sin(\alpha) g'(0) + \cos(\alpha) g(0) = 0; \, \tau_+ g\in L^2(\bbR_+;dx)\big\},      \no \\
& \hspace*{9.15cm} \alpha \in [0, \pi).  \no
\end{align}
The corresponding comparison operator with vanishing potential
coefficient $q \equiv 0$ will be denoted by $H_{+,\alpha}^{(0)}$,
$\alpha \in [0, \pi)$.

Next, introducing the Jost solutions
\begin{align}
& f_+ (z,x) = f_+^{(0)}(z,x) - \int_x^{\infty} dx' \, z^{-1/2} \sin(z^{1/2}(x-x'))
q(x') f_+(z,x'),    \lb{14.4.7} \\
& f_+^{(0)}(z,x) = e^{i z^{1/2} x}, \quad z \in \bbC, \; \Im\big(z^{1/2}\big) \geq 0,
\; x \geq 0,    \lb{14.4.8}
\end{align}
satisfying $\tau_+y = zy$, $z \in \bbC$, on $\bbR_+$,
and abbreviating $I_{L^2(\bbR_+; dx)} = I_+$, and
\begin{equation}
v = |q|^{1/2}, \; u = v \, \sgn(q), \, \text{ such that } \, q = uv = vu,   \lb{14.4.9}
\end{equation}
one infers the following facts:
\begin{align}
& {\det}_{L^2(\bbR_+;dx)} \Big(\ol{(H_{+,\alpha} - z I_+)^{1/2}
\big(H_{+,\alpha}^{(0)} - z I_+\big)^{-1} (H_{+,\alpha} - z I_+)^{1/2}}\Big)   \no \\
& \quad = {\det}_{L^2(\bbR_+;dx)}
\Big(I_+  + \ol{u \big(H_{+,\alpha}^{(0)} - z I_+\big)^{-1} v}\Big)
\lb{14.4.10} \\
& \quad = \f{\sin(\alpha) f_+'(z,0) + \cos(\alpha) f_+(z,0)}
{\sin(\alpha) i z^{1/2} + \cos(\alpha)}, \quad \alpha \in [0,\pi), \;
z \in \rho(H_{+,\alpha}^{(0)}) \cap \rho(H_{+,\alpha}).    \no
\end{align}
Here the first equality in \eqref{14.4.10} is shown as in the abstract context \eqref{14.2.8A}--\eqref{14.2.8AA}, and the second equality in \eqref{14.4.10} for the Dirichlet and Neumann cases $\alpha = 0$, $\alpha = \pi/2$ has
been discussed in \cite{GLMZ05}, \cite{GM03}, \cite{GMZ07}; the general case $\alpha \in [0,\pi)$ is proved in \cite[Theorem~2.6]{GN12}.

Since
\begin{equation}
\sigma_{ess}(H_{+,\alpha}) = [0, \infty), \quad \alpha \in [0, \pi),
\end{equation}
we now shift all operators $H_{+,\alpha}$ by $\lambda_1 I_+$, with a fixed
$\lambda_1 > 0$, and consider
\begin{equation}
H_{+,\alpha}(\lambda_1) = H_{+,\alpha} + \lambda_1 I_+, \quad
\alpha \in [0,\pi),
\end{equation}
from this point on and hence obtain
\begin{align}
& {\det}_{L^2(\bbR_+;dx)} \Big(\ol{(H_{+,\alpha}(\lambda_1) - z I_+)^{1/2}
\big(H_{+,\alpha}^{(0)}(\lambda_1) - z I_+\big)^{-1}
(H_{+,\alpha}(\lambda_1) - z I_+)^{1/2}}\Big)   \no \\
& \quad = {\det}_{L^2(\bbR_+;dx)}
\Big(I_+  + \ol{u \big(H_{+,\alpha}^{(0)}(\lambda_1) - z I_+\big)^{-1} v}\Big)
\no \\
& \quad = \f{\sin(\alpha) f_+'(z-\lambda_1,0)
+ \cos(\alpha) f_+(z-\lambda_1,0)}
{\sin(\alpha) i (z-\lambda_1)^{1/2}
+ \cos(\alpha)},      \lb{14.4.14} \\
& \hspace*{8mm}
\alpha \in [0,\pi), \;
z \in \rho\big(H_{+,\alpha}^{(0)}(\lambda_1)\big) \cap \rho(H_{+,\alpha}(\lambda_1)).
\no
\end{align}
In this half-line context all discrete eigenvalues
$H_{+,\alpha}(\lambda_1)$ (i.e., all eigenvalues of
$H_{+,\alpha}(\lambda_1)$ below $\lambda_1$) are simple and hence
\begin{equation}
m(0; H_{+,\alpha}(\lambda_1)) \in \{0,1\}, \quad \alpha \in [0,\pi).
\end{equation}
In addition, it is known that under Hypothesis \ref{14.h4.1}, the
threshold of the
essential spectrum of $H_{+,\alpha}(\lambda_1)$,
$\lambda_1$, is never an eigenvalue of $H_{+,\alpha}(\lambda_1)$.

Iterating the Volterra integral equation \eqref{14.4.7} for
$f_+(z-\lambda_1,0)$, and analogously for its $x$-derivative, yields
uniform asymptotic expansions near $z=0$ and as $z \to \infty$ (in
terms of powers of $|z|^{-1/2}$). The same applies to their $z$-derivatives
(cf., e.g., \cite[Ch.~I]{CS89}) and explicit computations yield the following.
For fixed $0 < \varepsilon_0$ sufficiently small, and using the abbreviation
$\bbC_{\varepsilon_0} = \bbC \backslash B(\lambda_1; \varepsilon_0)$,
with $B(z_0; r_0)$ the open ball in $\bbC$ of radius $r_0 > 0$ centered at $z_0 \in \bbC$, one obtains
\begin{align}
& f_+(z - \lambda_1,0)
\underset{\substack{|z|\to\infty \\ z \in \bbC_{\varepsilon_0}}}{=}
1 - \f{i}{2 z^{1/2}} \int_0^{\infty} dx_1 \, \big[e^{2i(z - \lambda_1)^{1/2}x_1}
- 1\big] q(x_1) + \Oh\big(|z|^{-1}\big),   \no \\
& f_+'(z - \lambda_1,0)
\underset{\substack{|z|\to\infty \\ z \in \bbC_{\varepsilon_0}}}{=}
i z^{1/2} - \f{1}{2} \int_0^{\infty} dx_1 \, \big[e^{2i(z - \lambda_1)^{1/2}x_1}
+ 1\big] q(x_1) + \Oh\big(|z|^{-1/2}\big),    \no \\
& \dot f_+(z - \lambda_1,0)
\underset{\substack{|z|\to\infty \\ z \in \bbC_{\varepsilon_0}}}{=}
\f{1}{2 z} \int_0^{\infty} dx_1 \, e^{2i(z - \lambda_1)^{1/2}x_1}
x_1 q(x_1) + \Oh\big(|z|^{-3/2}\big),    \no \\
& \dot f_+ \phantom{f}\hspace{-.45cm}' \,\,(z - \lambda_1,0)
\underset{\substack{|z|\to\infty \\ z \in \bbC_{\varepsilon_0}}}{=}
\f{i}{2 z^{1/2}}
- \f{i}{2 z^{1/2}} \int_0^{\infty} dx_1 \, e^{2i(z - \lambda_1)^{1/2}x_1}
x_1 q(x_1) + \Oh\big(|z|^{-1}\big)   \lb{14.4.16}
\end{align}
(abbreviating again $\dot{} = d/dz$).

Given the asymptotic expansions \eqref{14.4.16} as $|z| \to \infty$, and
employing the fact that the functions $f_+(\, \cdot - \lambda_1,0),
f_+'(\, \cdot - \lambda_1,0),
\dot f_+(\, \cdot - \lambda_1,0), {\dot f} _+\phantom{f}\hspace{-.45cm}'\,\,(\, \cdot - \lambda_1,0)$ are all analytic with respect to $z$ around $z=0$, investigating the case distinctions
$\alpha \in [0, \pi) \backslash\{0, \pi/2\}$, $\alpha =0$, $\alpha = \pi/2$,
$f_+(z - \lambda_1,0) \neq 0$, $f_+(z - \lambda_1,0) = 0$,
$f_+'(z - \lambda_1,0) \neq 0$, $f_+'(z - \lambda_1,0) = 0$, etc., one
verifies in each case that the logarithmic $z$-derivative of \eqref{14.4.10} satisfies the hypotheses in Theorem \ref{14.t2.8}, hence the latter applies
with $\varepsilon = 1/2$ to the pairs
$\big(H_{+,\alpha}^{(0)}(\lambda_1), H_{+,\alpha}(\lambda_1)\big)$,
$\alpha \in [0, \pi)$.

More generally, we now replace the pair $(0,q)$ by $(q_1,q_2)$, where $q_j$,
$j=1,2$, satisfy Hypothesis \ref{14.h4.1}, and denote the
corresponding Schr\"odinger operators in $L^2(\bbR_+;dx)$ with $q$
(resp., $u$, $v$) replaced by $q_j$ (resp., $u_j$, $v_j$) by
$H_{+,\alpha,j}$ and similarly, after the shift with $\lambda_1$, by
$H_{+,\alpha,j}(\lambda_1)$, $j=1,2$. Analogously, we denote the
corresponding Jost solutions by $f_{+,j}(z, \, \cdot \,)$, $j=1,2$. This then
yields the following results:

%%%%%%%
\begin{theorem} \lb{14.t4.2}
Suppose $q_j$, $j=1,2$, satisfy Hypothesis \ref{14.h4.1}. Then,
\begin{align}
& {\det}_{L^2(\bbR_+;dx)} \Big(\ol{(H_{+,\alpha,2}(\lambda_1) - z I_+)^{1/2}
\big(H_{+,\alpha,1}(\lambda_1) - z I_+\big)^{-1}
(H_{+,\alpha,2}(\lambda_1) - z I_+)^{1/2}}\Big)   \no \\
& \quad = {\det}_{L^2(\bbR_+;dx)}
\Big(I_+  + \ol{u_{1,2}
\big(H_{+,\alpha,1}(\lambda_1) - z I_+\big)^{-1} v_{1,2}}\Big)
\no \\
& \quad = \f{\sin(\alpha) f_{+,2}'(z-\lambda_1,0)
+ \cos(\alpha) f_{+,2}(z-\lambda_1,0)}
{\sin(\alpha) f_{+,1}'(z-\lambda_1,0)
+ \cos(\alpha) f_{+,1}(z-\lambda_1,0)},  \lb{14.4.17} \\
& \hspace*{7.2mm}
\alpha \in [0,\pi), \;
z \in \rho\big(H_{+,\alpha ,1}(\lambda_1)\big) \cap \rho(H_{+,\alpha ,2}(\lambda_1)),
\no
\end{align}
where
\begin{align}
\begin{split}
& v_{1,2} = |q_2 - q_1|^{1/2}, \; u_{1,2} = v_{1,2} \, \sgn(q_2 - q_1),   \\
& \quad \text{such that } \, q_2 - q_1 =  u_{1,2} v_{1,2} = v_{1,2}  u_{1,2}.  \lb{14.4.18}
\end{split}
\end{align}
In addition,
\begin{align}
& \zeta(s; H_{+,\alpha,1}(\lambda_1),H_{+,\alpha,2}(\lambda_1))
= e^{is (\pi - \theta)} \pi^{-1} \sin(\pi s) \no \\
& \quad \times \int_0^{\infty} dt \, t^{-s}
\f{d}{dt} \ln\bigg((e^{i\theta}t)^{[m(0;H_{+,\alpha,1}(\lambda_1))
- m(0;H_{+,\alpha,2}(\lambda_1))]}   \lb{14.4.19} \\
& \hspace*{3.2cm} \times
\f{\sin(\alpha) f_{+,2}'(e^{i \theta} t - \lambda_1,0)
+ \cos(\alpha) f_{+,2}(e^{i \theta} t - \lambda_1,0)}
{\sin(\alpha) f_{+,1}'(e^{i \theta} t - \lambda_1,0)
+ \cos(\alpha) f_{+,1}(e^{i \theta} t - \lambda_1,0)}\bigg),  \no \\
& \hspace*{7.1cm} \alpha \in [0,\pi), \; \Re(s) \in (-1/2,1).   \no
\end{align}
\end{theorem}
%%%%%%%
\begin{proof}
Relations \eqref{14.4.17}, \eqref{14.4.18} follow as summarized in \eqref{14.4.7}--\eqref{14.4.16}
and the two paragraphs preceding Theorem \ref{14.t4.2}. Thus, Theorem \ref{14.t2.8} applies to
the pairs of self-adjoint operators
$(H_{+,\alpha,1}(\lambda_1), H_{+,\alpha,2}(\lambda_1))$, $\alpha \in [0, \pi)$.
\end{proof}
%%%%%%%

The relative $\zeta$-function regularized determinant now follows immediately from Theorem \ref{14.t2.10}.

Special cases of \eqref{14.4.17} (pertaining to the Dirichlet boundary conditions $\alpha_j=0$, $j=1,2$)
appeared in the celebrated work by Jost and Pais \cite{JP51} and Buslaev and Faddeev \cite{BF60}
(see also \cite{DU11}, \cite{GM03}, \cite{Ne80}, \cite{OY12}, \cite{RS97}).

Up to this point we kept the boundary condition, that is, $\alpha$, fixed and varied the potential coefficient $q$. Next, we keep $q$ fixed, but vary
$\alpha$. Returning to the operator $H_{+,\alpha}$, $\alpha \in [0, \pi)$,
we turn to its underlying quadratic form $Q_{H_{+,\alpha}}$ next,
\begin{align}
& Q_{H_{+,\alpha}}(f,g)
= \int_0^{\infty} dx \big[\ol{f'(x)} g'(x) + q(x) \ol{f(x)} g(x)\big]
-  \cot(\alpha) \ol{f(0)} g(0),   \no \\
&  f, g \in \dom(Q_{H_{+,\alpha}}) = \dom\big(|H_{+,\alpha}|^{1/2}\big) = H^1(\bbR_+)      \lb{14.4.20} \\
& \quad = \big\{h\in L^2(\bbR_+; dx) \,|\,
h \in AC ([0,b]) \, \text{for all $b>0$}; \, h' \in L^2(\bbR_+; dx)\big\},   \no \\
& \hspace*{8.95cm}  \alpha \in (0,\pi),   \no \\
& Q_{H_{+,0}}(f,g) = \int_0^{\infty} dx \big[\ol{f'(x)} g'(x)
+ q(x) \ol{f(x)} g(x)\big],    \no \\
& f, g \in \dom(Q_{H_{+,0}}) = \dom\big(|H_{+,0}|^{1/2}\big)
= H^1_0(\bbR_+)     \lb{14.4.21} \\
& \quad = \big\{h\in L^2(\bbR_+; dx) \,|\,
h \in AC([0,b]) \, \text{for all $b>0$}; \, h(0)=0;
\, h' \in L^2(\bbR_+; dx)\big\}.  \no
\end{align}
Moreover, introducing the regular solution $\phi_{\alpha}(z, \, \cdot \,)$
associated with $H_{+,\alpha}$ satisfying $\tau_+ y = z y$, $z \in \bbC$, on $\bbR_+$, and
\begin{equation}
\sin(\alpha) \phi_{\alpha}'(z,0) + \cos(\alpha) \phi_{\alpha}(z,0) = 0,
\quad \alpha \in [0,\pi), \; z \in \bbC,   \lb{14.4.22}
\end{equation}
one infers
\begin{align}
\begin{split}
\phi_{\alpha}(z,x)=\phi_{\alpha}^{(0)}(z,x)
+ \int_0^x dx' \, z^{-1/2} \sin(z^{1/2}(x-x')) q(x')
\phi_{\alpha}(z,x'),&    \lb{14.4.23} \\
z \in \bbC, \; \Im(z^{1/2}) \geq 0, \; x \geq 0,&
\end{split}
\end{align}
with
\begin{align}
\begin{split}
\phi_{\alpha}^{(0)}(z,x) = \cos(\alpha) z^{-1/2} \sin(z^{1/2}x)
- \sin(\alpha)\cos(z^{1/2}x),&  \\
z \in \bbC, \; \Im(z^{1/2}) \geq 0, \; x \geq 0.&
\end{split}
\end{align}
Given the solutions $\phi_{\alpha}(z, \, \cdot \,)$, $f_+(z, \, \cdot \,)$ of
$\tau_+ y = z$, $z \in \bbC$, the resolvent of
$H_{+,\alpha}$ is given by
\begin{align}
\begin{split}
\big((H_{+,\alpha}-z I_+)^{-1}f \big)(x)
=\int_0^{\infty}dx' \,
G_{+,\alpha}(z,x,x')f(x'),&    \lb{14.4.25} \\
z\in \rho(H_{+,\alpha}), \; x \geq 0,
\; f \in L^2(\bbR_+; dx),&
\end{split}
\end{align}
with the Green's function $G_{+,\alpha}$ of $H_{+,\alpha}$
expressed in terms of $\phi_{\alpha}$ and $f_+$ by
\begin{align}
\begin{split}
& G_{+,\alpha}(z,x,x')
= (H_{+,\alpha} - z I_+)^{-1}(x,x')    \lb{14.4.26} \\
& \quad = \frac{1}{W(f_+(z,\cdot), \phi_{\alpha}(z,\cdot))} \begin{cases}
\phi_{\alpha}(z,x) \, f_+(z,x'), & 0\leq x\leq x' < \infty, \\
\phi_{\alpha}(z,x') \, f_+(z,x), & 0\leq x' \leq x < \infty, \end{cases} \\
& \hspace*{7.93cm} z \in \rho(H_{+,\alpha}).
\end{split}
\end{align}
In the special case $\alpha = 0$ one verifies
\begin{equation}
W(f_+(z,\cdot), \phi_{0}(z,\cdot)) = f_+(z,0), \quad z \in \rho(H_{+,0}),  \lb{14.4.26a}
\end{equation}
with $f_+(z,0)$ the well-known Jost function.

Next, we compare the half-line Green's functions $G_{+,\alpha_1}$
and $G_{+,\alpha_2}$, that is, we investigate the integral kernels
associated with a special case of Krein's formula for resolvents
(cf.\ \cite[\& 106]{AG81}): Assume $\alpha_1,\alpha_2 \in [0,\pi)$, with
$\alpha_1 \neq \alpha_2$. Then,
\begin{align}
\begin{split}
&G_{+,\alpha_2}(z,x,x')=G_{+,\alpha_1}(z,x,x')
- \f{\psi_{+,\alpha_1}(z,x) \psi_{+,\alpha_1}(z,x')}{\cot(\alpha_2 - \alpha_1)
+ m_{+,\alpha_1}(z)}
,      \lb{14.4.27} \\
&\hspace*{3.1cm}
z \in \rho(H_{+,\alpha_1})\cap\rho(H_{+,\alpha_2}),
\; x,x' \in [0,\infty),
\end{split}
\end{align}
and
\begin{equation}
\int_0^{\infty} dx \, \psi_{+,\alpha}(z,x)^2 = \dot m_{+,\alpha}(z),
\end{equation}
implying
\begin{align}
\begin{split}
& {\tr}_{L^2(\bbR_+; dx)} \big((H_{+,\alpha_2}(\lambda_1) - z I_+)^{-1}
- (H_{+,\alpha_1}(\lambda_1) - z I_+)^{-1}\big)     \\
& \quad = - \f{d}{dz} \ln[\cot(\alpha_2 - \alpha_1) + m_{+,\alpha_1}(z)], \quad
z \in \rho(H_{+,\alpha_1})\cap\rho(H_{+,\alpha_2}),    \lb{14.4.29}
\end{split}
\end{align}
according to Lemma~2.2 and (A.44) in \cite{GS96}. Here $\psi_{+,\alpha}(z, \, \cdot \,)$ and
$m_{+,\alpha}(z)$ are the Weyl--Titchmarsh solution and $m$-function corresponding to
$H_{+,\alpha}$. More precisely,
\begin{equation}
\psi_{+,\alpha}(z, \, \cdot \,) = \theta_{\alpha} (z, \, \cdot \,) + m_{+,\alpha}(z) \phi_{\alpha} (z, \, \cdot \,),
\quad z \in \rho(H_{+,\alpha}),    \lb{14.4.30}
\end{equation}
where
\begin{align}
\begin{split}
\theta_{\alpha}(z,x)=\theta_{\alpha}^{(0)}(z,x)
+ \int_0^x dx' \, z^{-1/2} \sin(z^{1/2}(x-x')) q(x')
\theta_{\alpha}(z,x'),&    \lb{14.4.31} \\
z \in \bbC, \; \Im(z^{1/2}) \geq 0, \; x \geq 0,&
\end{split}
\end{align}
with
\begin{align}
\begin{split}
\theta_{\alpha}^{(0)}(z,x) = \cos(\alpha) \cos(z^{1/2} x) + \sin(\alpha)
z^{-1/2} \sin(z^{1/2} x),&   \\
z \in \bbC, \; \Im(z^{1/2}) \geq 0, \; x \geq 0.&
\end{split}
\end{align}

Due to the limit point property of $\tau_+$ at $+\infty$, one actually has
\begin{equation}
\psi_{+,\alpha}(z,\, \cdot \,) = C_{\alpha}(z) f_+(z,\, \cdot \,), \quad z \in \rho(H_{+,\alpha}),
\end{equation}
for some $z$-dependent complex-valued constant $C_{\alpha}(z)$. Actually,
since $\psi_{+,\alpha} (z,0) = C_\alpha f_+ (z,0)$, one can
show (using (A.18) in \cite{GS96}) that
\begin{equation}
C_\alpha (z) = \frac{\cos (\alpha) - \sin(\alpha)  m_{+,\alpha}(z) }{f_+ (z,0)} = \frac 1 {\cos (\alpha) f_+ (z,0) + \sin(\alpha)  f_+ ' (z,0)}.
\end{equation}
Similarly,
\begin{equation}
m_{+,0}(z) = \psi'_{+,0} (z,0)/\psi_{+,0} (z,0) = f'_+(z,0)/f_+(z,0), \quad z \in \rho(H_{+,\alpha}),
\end{equation}
and
\begin{align}
\begin{split}
m_{+,\alpha}(z) = \f{- \sin(\alpha) + \cos(\alpha) m_{+,0}(z)}{\cos(\alpha) + \sin(\alpha) m_{+,0}(z)}
=\f{\cos(\alpha) f_+'(z,0) - \sin(\alpha) f_+(z,0)}{\sin(\alpha) f_+'(z,0) + \cos(\alpha) f_+(z,0)},& \\
z \in \rho(H_{+,\alpha}).&
\end{split}
\end{align}
Moreover, one verifies that
\begin{align}
\cos(\alpha_2 - \alpha_1) + \sin(\alpha_2 - \alpha_1) m_{+,\alpha_1}(z)
=\f{\sin(\alpha_2) f_+'(z,0) + \cos(\alpha_2) f_+(z,0)}{\sin(\alpha_1) f_+'(z,0) + \cos(\alpha_1) f_+(z,0)}.
\lb{14.4.33}
\end{align}
Combining \eqref{14.4.29} and \eqref{14.4.33} yields
\begin{align}
& {\tr}_{L^2(\bbR_+; dx)} \big((H_{+,\alpha_2}(\lambda_1) - z I_+)^{-1}
- (H_{+,\alpha_1}(\lambda_1) - z I_+)^{-1}\big)    \no \\
& \quad = - \f{d}{dz} \ln\bigg(\f{\sin(\alpha_2) f_+'(z - \lambda_1,0)
+ \cos(\alpha_2) f_+(z - \lambda_1,0)}{\sin(\alpha_1) f_+'(z - \lambda_1,0)
+ \cos(\alpha_1) f_+(z - \lambda_1,0)}\bigg),    \lb{14.4.34} \\
& \hspace*{1.6cm}
z \in \rho(H_{+,\alpha_1}(\lambda_1))\cap\rho(H_{+,\alpha_2}(\lambda_1)),
\; \alpha_1, \alpha_2 \in [0,\pi).    \no
\end{align}
In particular,
$\big[(H_{+,\alpha_2} - z I_+)^{-1} - (H_{+,\alpha_1} - z I_+)^{-1}\big]$
is rank-one and hence trace class. Moreover, since by \eqref{14.4.20}
\begin{equation}
f_+(z, \, \cdot \,) \in \dom\big(|H_{+,\alpha}|^{1/2}\big) = H^1(\bbR_+),
\quad \alpha \in (0,\pi),     \lb{14.4.35}
\end{equation}
Hypothesis \ref{14.h2.4a} and hence relations
\eqref{14.2.23a}, \eqref{14.2.24a} are now satisfied for the pair
$(H_{+,\alpha_1}(\lambda_1), H_{+,\alpha_2}(\lambda_1))$ for
$\alpha_1 \in [0,\pi)$, $\alpha_2 \in (0,\pi)$, implying the following result:

%%%%%%%
\begin{theorem} \lb{14.t4.3}
Suppose $q_j$, $j=1,2$, satisfy Hypothesis \ref{14.h4.1}. Then,
\begin{align}
& {\tr}_{L^2(\bbR_+; dx)} \big((H_{+,\alpha_2}(\lambda_1) - z I_+)^{-1}
- (H_{+,\alpha_1}(\lambda_1) - z I_+)^{-1}\big)    \no \\
& \quad = - \f{d}{dz} \ln\Big(
{\det}_{L^2(\bbR_+;dx)} \Big(\big\{(H_{+,\alpha_2}(\lambda_1) - z I_+)^{1/2}
\big(H_{+,\alpha_1}(\lambda_1) - z I_+\big)^{-1}    \no \\
& \hspace*{4.15cm} \times
(H_{+,\alpha_2}(\lambda_1) - z I_+)^{1/2}\big\}^c\Big)\Big)  \lb{14.4.36}  \\
& \quad = - \f{d}{dz} \ln\bigg(\f{\sin(\alpha_2) f_+'(z - \lambda_1,0)
+ \cos(\alpha_2) f_+(z - \lambda_1,0)}{\sin(\alpha_1) f_+'(z - \lambda_1,0)
+ \cos(\alpha_1) f_+(z - \lambda_1,0)}\bigg),    \lb{14.4.37} \\
& \quad \;
z \in \rho(H_{+,\alpha_1}(\lambda_1))\cap\rho(H_{+,\alpha_2}(\lambda_1)),
\; \alpha_1 \in [0,\pi), \; \alpha_2 \in (0,\pi), \; \alpha_1 \neq \alpha_2,   \no
\end{align}
temporarily abbreviating the operator closure symbol by $\{\cdots\}^c$
due to lack of space in \eqref{14.4.36}. In addition,
\begin{align}
& \zeta (s; H_{+,\alpha_1} (\lambda_1) , H_{+, \alpha_2} (\lambda_1)) = e^{is (\pi -\theta )} \pi^{-1}
\sin (\pi s) \no \\
& \quad \times \int_0^\infty dt \, t^{-s} \frac d {dt} \ln \left( (t e^{i\theta}) ^{[ m(0,H_{+,\alpha_1} (\lambda_1))-m(0,H_{+,\alpha_2} (\lambda_1))]} \right.    \lb{14.4.38} \\
& \quad \times \left. \frac{\sin (\alpha_2) f_+' (te^{i\theta} - \lambda_1,0) + \cos (\alpha _2) f_+ (te^{i\theta} -\lambda_1,0)}
{\sin (\alpha_1) f_+' (te^{i\theta} - \lambda_1,0) + \cos (\alpha _1) f_+ (te^{i\theta} -\lambda_1,0)} \right) , \no \\
&\hspace*{2.6cm} \alpha_j \in (0,\pi), \; j=1,2, \,\,\Re(s) \in (-1/2,1).   \no
\end{align}
\end{theorem}
%%%%%%%%
\begin{proof}
Relations \eqref{14.4.36}, \eqref{14.4.37} summarize the discussion in \eqref{14.4.22}--\eqref{14.4.35}.
Applying Theorem \ref{14.t2.8} in the case $\alpha_1, \alpha _2 \in (0,\pi)$ then yields
\eqref{14.4.38}.
\end{proof}
%%%%%%%%

The relative $\zeta$-function regularized determinant again follows immediately from Theorem \ref{14.t2.10},
\begin{align}
& \zeta ' (0; H_{+,\alpha_1} (\lambda_1) , H_{+,\alpha_2} (\lambda_1)) =
- \lim_{t \downarrow 0} \f{d}{dt} \ln\bigg(\big(t e^{i \theta}\big)^{[m(0,H_{+,\alpha_1} (\lambda_1))
- m(0,H_{+,\alpha_2} (\lambda_1))]}     \no \\[1mm]
& \quad \times \f{\sin(\alpha_2) f_+' (t e^{i \theta} - \lambda_1,0)
+ \cos(\alpha_2) f_+(t e^{i \theta} - \lambda_1,0)}{\sin(\alpha_1) f_+' (t e^{i \theta} - \lambda_1,0)
+ \cos(\alpha_1) f_+(t e^{i \theta} - \lambda_1,0)}\bigg).
\end{align}

%%%%%%%%
\begin{remark}
The case $\alpha_1 =0$, $\alpha_2 \in (0,\pi )$, is more involved in that the representation
\eqref{14.4.38}  is only valid for $\Re (s) \in (0,1)$. This is due to the fact that
\begin{equation}
\f{\sin(\alpha_2) f_+'(z - \lambda_1,0) + \cos(\alpha_2) f_+(z - \lambda_1,0)}{f_+(z - \lambda_1,0)}
\underset{|z|\to\infty}{=} \Oh\big(|z|^{1/2}\big),
\end{equation}
and hence assumption \eqref{14.2.31} is not satisfied.
One then proceeds as follows.
 Let $H_{+,\alpha}^{(0)} (\lambda_1)$ denote the case with vanishing potential $q$. We rewrite
\begin{align}
& {\tr}_{L^2(\bbR_+; dx)} \big((H_{+,\alpha_2}(\lambda_1) - z I_+)^{-1}
- (H_{+,0}(\lambda_1) - z I_+)^{-1}\big)    \no \\
& \quad ={\tr}_{L^2(\bbR_+; dx)} \big((H_{+,\alpha_2}(\lambda_1) - z I_+)^{-1}
- (H_{+,\alpha_2}^{(0)} (\lambda_1) - z I_+)^{-1}\big)    \no \\
& \qquad +{\tr}_{L^2(\bbR_+; dx)} \big((H_{+,0}^{(0)} (\lambda_1) - z I_+)^{-1}
- (H_{+,0} (\lambda_1) - z I_+)^{-1}\big)    \no \\
& \qquad +{\tr}_{L^2(\bbR_+; dx)} \big((H_{+,\alpha_2}^{(0)} (\lambda_1) - z I_+)^{-1}
- (H_{+,0}^{(0)} (\lambda_1) - z I_+)^{-1}\big).     \lb{14.4.46}
\end{align}
For the first two of the three
contributions on the right-hand side of \eqref{14.4.46} the relative $\zeta$-determinants
can be computed from Theorem \ref{14.t4.2}. For the third contribution more care is needed as the subleading large-$|z|$ behavior differs
due to one boundary condition being Dirichlet and the other being Robin. The starting point is the representation
\begin{align}
\begin{split}
& \zeta \Big(s; H_{+,0}^{(0)} (\lambda_1),H_{+,\alpha_2}^{(0)}(\lambda_1)\Big)
= e^{is (\pi - \theta)} \pi^{-1} \sin(\pi s)       \\
& \quad \times \int_0^\infty dt \, t^{-s} \frac d {dt} \ln \left(\sin (\alpha_2)
i \big(t e^{i\theta} - \lambda_1\big)^{1/2} + \cos (\alpha _2)\right).
\end{split}
\end{align}
Along the lines of \eqref{14.3.subasym} one rewrites this as
\begin{align}
\begin{split}
& \zeta \Big(s; H_{+,0}^{(0)} (\lambda_1), H_{+,\alpha_2}^{(0)} (\lambda_1)\Big)  = e^{is (\pi - \theta )} \pi ^{-1} \sin (\pi s) \int_0^\infty dt \, t^{-s}     \\
& \quad \times \frac d {dt} \ln \left(\frac{\sin (\alpha _2) i \big(t e^{i\theta } - \lambda_1\big)^{1/2} + \cos (\alpha _2)}{\sin (\alpha_2) i \big(t e^{i\theta} - \lambda_1\big)^{1/2} }\right) + \zeta^{asym} (s),
\end{split}
\end{align}
where, using \eqref{14.3.intform},
\begin{align}
\zeta ^{asym} (s) &= e^{is (\pi - \theta )} \pi^{-1} \sin (\pi s) \int_0^\infty dt \, t^{-s} \frac d {dt} \ln \left(i \sin (\alpha _2) \big(te^{i\theta} - \lambda_1\big)^{1/2} \right)    \no \\
& = \lambda_1^{-s}/2.
\end{align}
The relative $\zeta$-determinant for the last term in \eqref{14.4.46} then follows from
\begin{align}
\zeta ' \Big(0; H_{+,0}^{(0)} (\lambda_1) , H_{+,\alpha_2}^{(0)} (\lambda_1)\Big) &= \Re \left( - \ln  \bigg(\frac{ \sin (\alpha_2) i (- \lambda_1)^{1/2} + \cos (\alpha_2)}{ \sin (\alpha_2) i (-\lambda_1)^{1/2}}\bigg)
- \ln \big(\lambda_1^{1/2}\big)\right) \nonumber\\
& = - \ln \big(\big|\lambda_1^{1/2} -\cot (\alpha_2)\big|\big).
\end{align}
${}$ \hfill $\diamond$
\end{remark}
%%%%%%%%

One notes that formally, \eqref{14.4.37} extends to the trivial case $\alpha_1 = \alpha_2$.

For the case of a strongly singular potential on the half-line with $x^{-2}$-type singularity at $x=0$
we refer to \cite{KLP06}.

%%%%%%%%%%%%%%%%%%%%%%%%%
\medskip
\noindent {\bf Acknowledgments.}
We are indebted to Christer Bennewitz for very helpful discussions.
%%%%%%%%%%%%%%%%%%%%%%%%%

%%%%%%%%%%%%%%%%%%%%%%%%%%%%
%%%%%%%%%%%%%%%%%%%%%%%%%%%%

\end{document}